\documentclass{amsart}
\usepackage{amsmath, amssymb, amsthm}
\usepackage[hypertex,dvips]{hyperref}
\usepackage{hypbmsec}
\usepackage{graphicx}

\newtheorem{prop}{Proposition}[section]
\newtheorem{cor}[prop]{Corollary}
\newtheorem{thm}[prop]{Theorem}
\newtheorem{lemma}[prop]{Lemma}

\newtheorem{conj}[prop]{Conjecture}

\theoremstyle{definition}
\newtheorem{defn}[prop]{Definition}

\theoremstyle{remark}

\newtheorem{rem}[prop]{Remark}

\DeclareMathOperator\Spec{Spec}

\DeclareMathOperator\Gal{Gal}
\DeclareMathOperator\Cl{Cl}
\DeclareMathOperator\Sqrt{Sqrt}
\DeclareMathOperator\Sqct{SqCt}
\DeclareMathOperator\ram{ram}

\DeclareMathOperator\Out{Out}
\newcommand{\BAR}[1]{{\overline{\mathbb{#1}}}}
\newcommand{\GQ}{\Gal(\overline{\mathbb{Q}}/\mathbb{Q})}

\newcommand{\etpi}[1]{\pi_1^{\text{\'{e}t}}\left(#1\right)}
\newcommand{\base}{\overrightarrow{01}}
\newcommand{\FreeTwo}{{F_2}}
\newcommand{\GT}{\widehat{GT}}

\usepackage[small,nohug,heads=vee,PostScript]{diagrams} \diagramstyle[labelstyle=\scriptstyle]

\usepackage{calc,enumitem}
\newenvironment{casework}{
    \begin{enumerate}[labelindent=15pt,labelwidth=\widthof{Case 9: },leftmargin=!]
    }
{\end{enumerate}}
\newenvironment{subcasework}{
    \begin{enumerate}[labelindent=-15pt,labelwidth=\widthof{Subcase 9.9: },leftmargin=!]
    }
{\end{enumerate}}

\title{A new \texorpdfstring{$\GQ$-}{Galois }invariant of dessins d'enfants}
\author{Ravi Jagadeesan}
\address{Harvard College, 1 Oxford St, Cambridge, MA 02138}
\email{ravi.jagadeesan@gmail.com}

\begin{document}

\begin{abstract}
We study the action of $\GQ$ on the category of Belyi functions
(finite, \'{e}tale covers of $\mathbb{P}^1_{\BAR{Q}}\setminus \{0,1,\infty\}$).
We describe a new combinatorial $\GQ$-invariant for Belyi functions whose monodromy cycle types above $0$ and $\infty$ are the same.  We use a version of our invariant to prove that $\GQ$ acts faithfully on the set of Belyi functions whose monodromy cycle types above 0 and $\infty$ are the same; the proof of this result involves a version of Belyi's Theorem for odd degree morphisms.
Using our invariant, we obtain that for all $k < 2^{\sqrt{\frac{2}{3}}}$ and
all positive integers $N$, there is an $n \le N$ such that
the set of degree $n$ Belyi functions of a particular rational Nielsen class must split
into at least $\Omega\left(k^{\sqrt{N}}\right)$ Galois orbits.
\end{abstract}

\maketitle


\section{Introduction}
In his \emph{Esquisse d'un Programme}~\cite{Esq}, Grothendieck described a research
program to understand the structure of $\GQ$.  One idea is that there is a faithful,
outer action of $\GQ$ on the Teichm\"{u}ller tower of profinite mapping class groups (the \'{e}tale fundamental
groups of the moduli spaces $M_{g,n}$ of curves of genus $g$ with $n$ ordered marked points
over $\BAR{Q}$).  Grothendieck conjectured that the group of outer automorphisms of the Teichm\"{u}ller tower is in fact isomorphic to $\GQ$, and that the action is ``generated" on the dimension 1 moduli spaces with
``relations" in dimension 2.  The moduli space $M_{0,4}$ is of dimension 1,
and is isomorphic to $\mathbb{P}^1_\BAR{Q} \setminus \{0,1,\infty\}$, and therefore as part of the program,
one wishes to study the action of $\GQ$ on the category of \'{e}tale covers of
$\mathbb{P}^1_\BAR{Q} \setminus \{0,1,\infty\}$.  Grothendieck's
\emph{dessins d'enfants} encode the covers combinatorially,
and one can try to understand the faithful action of $\GQ$ on them.  A first step is to determine a set of
invariants, perhaps algebraic, arithmetic, geometric, or topological in nature,
that can distinguish distinct $\GQ$-orbits of dessins.  In this paper, we construct
a new invariant for Belyi functions whose monodromy cycle types above 0 and $\infty$ are the same.

The key idea is to consider commutative squares of the form
\[\begin{diagram}
Y & \lTo & X\\
\dTo & & \dTo\\
\mathbb{P}^1 & \lTo_{t=\frac{4z}{(z+1)^2}} & \mathbb{P}^1
\end{diagram}\]
with $X$ the normalization of the fibered product $Y \times_{\mathbb{P}^1} \mathbb{P}^1$.
In certain cases, $\GQ$-invariants of the left morphism extend to $\GQ$-invariants of the right morphism.
By considering the cycle types of the monodromy generators
of the left morphism as a $\GQ$-invariant, we partition the set of possible right
morphisms into $\GQ$-invariant subsets.  We describe this new invariant combinatorially
as the \emph{square-root cycle type class}.
It can help distinguish $\GQ$-orbits of Belyi functions
that have the same monodromy cycle type over $0$ and $\infty$.  In
Theorems~\ref{thm:ClLowerBound} and~\ref{thm:Cl'LowerBound}, we prove that
our invariant is substantially finer than the rational Nielsen class (and therefore
substantially finer than the monodromy group and the monodromy cycle type).
In particular, we prove that for all $k < 2^{\sqrt{\frac{2}{3}}}$ and
all positive integers $N$, there is an $n \le N$ such that
the set of degree $n$ Belyi functions of a particular rational Nielsen class must split
into at least $\Omega\left(k^{\sqrt{N}}\right)$ Galois orbits.

By varying $Y$ over curves of genus 1 in an appropriate manner, we establish in Corollary~\ref{cor:GQFaithfulOn0InftySames} that the action of $\GQ$ is faithful on the set of Belyi functions whose monodromies above 0 and $\infty$ are the same.  The proof uses the properties of our $\GQ$-invariant of Belyi functions.  An intermediate step requires us to construct odd-degree Belyi functions, which we do in Theorem~\ref{thm:OddDegreeBelyi} by adjusting Belyi's first proof of his celebrated theorem.

Nakamura and Schneps~\cite[Theorem 2.2]{NakamuraSchneps} derived a constraint on the image of $\GQ$ in $\GT$ using the fact that $t = \frac{4f}{(f+1)^2}$ is defined over $\mathbb{Q}$. Our commutative squares can be reinterpreted as pulling back
\'{e}tale covers of a genus 0 smooth one-dimensional Deligne-Mumford
stack to $\mathbb{P}^1 \setminus \{0,1,\infty\}$, and therefore can be considered a reinterpretation of \cite[Theorem 2.2]{NakamuraSchneps}.

The structure of this paper is as follows.  In Section~\ref{sec:Previous},
we recall the basic definitions and discuss previous work.  In Section~\ref{sec:Statements},
we state our main results, and in Section~\ref{sec:ProofProp}, we prove the basic properties
of our new invariant. In Section~\ref{sec:ProofMainLower}, we prove that our invariant is stronger than the rational
Nielsen class invariant in certain cases.  In Section~\ref{sec:FaithfullnessProofs}, we prove that the action of $\GQ$ is faithful on the class of Belyi functions under consideration, and in Section~\ref{sec:Conclusion}, we
give concluding remarks and state an open problem.  Elementary computations are deferred
to Appendix~\ref{app:CompLemmata}.

\subsection*{Acknowledgements}
This research was done in the MIT Math Department's PRIMES program.
The author would like to thank Akhil Mathew for his incredibly helpful
insight and guidance that influenced this work.
The author would also like to thank Noam Elkies for proposing
this project and offering numerous useful observations, such as
suggesting that we consider fibered products of curves, suggesting
the proof of Theorem~\ref{thm:PropOfSqct}(c), and suggesting that we apply
Proposition~\ref{prop:p,22}. The author would also like to Pavel Etingof and Kirsten Wickelgren for helpful discussions, as well as the anonymous referee for numerous helpful suggestions.

\section{Notation and Previous Work}
\label{sec:Previous}
Unless otherwise specified,
a curve will mean a smooth, irreducible, projective, algebraic curve over $\mathbb{C}$, or equivalently
a compact Riemann surface.  We will denote
by $\mathbb{P}^1$ the complex projective line $\mathbb{P}^1_{\mathbb{C}}$.
Fix an embedding $\BAR{Q} \hookrightarrow \mathbb{C}$.

Fundamental groups are topological unless otherwise specified.
We fix a generating set $x_0,x_1,x_\infty$ of $\pi_1\left(\mathbb{P}^1\setminus \{0,1,\infty\},\base\right)$
such that $x_0x_1x_\infty = 1$ in Figure~\ref{fig:P1Minus3}: the loops have winding numbers of $1,0,-1$
about 0 and $0,1,-1$ about 1, respectively.
Sending the generators $x,y$ of $\FreeTwo$, the free group on two letters, to $x_0,x_1$, respectively, yields an isomorphism
$\FreeTwo \cong \pi_1\left(\mathbb{P}^1\setminus \{0,1,\infty\},\base\right)$.
\begin{figure}
\includegraphics[width=4in]{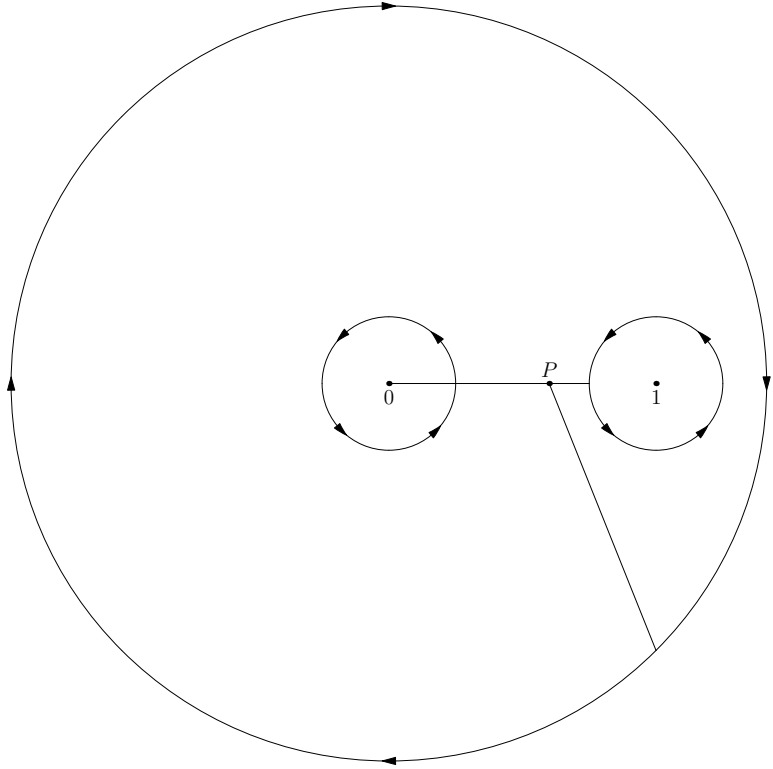}
\caption{{\bf Generators for $\pi_1\left(\mathbb{P}^1\setminus \{0,1,\infty\},\protect\base\right).$} The base-point is the tangent vector $\protect\base$ 
at 0.  The homotopy class $x_0$ is given by moving from 0 toward 1, in the counterclockwise around 0, and back to 0 along the segment between 0 and 1.
The homotopy class $x_1$ is defined similarly.  The homotopy class $x_\infty$ is defined by moving from 0 to $P$ along the segment, traversing the segment from $P$ to the large circle, moving around the large circle clockwise, returning to $P$, and then returning to 0 along the segment.  
It is evident that $x_0x_1x_\infty = 1$.
}
\label{fig:P1Minus3}
\end{figure}

By a \emph{weak action} of a group $G$ on a category $\mathcal{C}$, we mean a group
homomorphism from $G$ to the group of equivalences from $\mathcal{C}$ to $\mathcal{C}$,
modulo natural isomorphism.  Let $\widehat{G}$ denote the profinite completion of a group $G$.

For a positive integer $n$, let $[n] = \{1,2,\ldots,n\}$.
We write $\psi \dashv n$ if $\psi$ is a partition of $n$, by which we mean a non-increasing sequence of positive integers that sum to $n$ (for example, $(2,2,1) \dashv 5$).  Given $\psi \dashv n$, let the \emph{ramification number}
of $\psi$, which we denote by $\ram(\psi)$, equal $n - k$, where
$k$ is the number of non-empty parts of $\psi$ (i.e. the length of the sequence of positive integers).  We can extend the definition of $\ram$ to
permutations $\sigma \in S_n$ by defining the ramification number of $\sigma$ to be
the ramification number of the cycle type of $\sigma$.

\subsection{The \texorpdfstring{action of $\GQ$}{Galois action} on profinite fundamental groups}
Let $\overline{p}$ be a geometric (potentially tangential in the sense of Deligne~\cite[\S15]{DeligneP1Minus3}) point of $\mathbb{P}^1_\BAR{Q} \setminus \{0,1,\infty\}$, let
$p$ be the corresponding geometric point of $\mathbb{P}^1_\mathbb{Q}$,
and let $p_\mathbb{C}$ be the base-change of $\overline{p}$ to $\mathbb{C}$.
There is an isomorphism between \'{e}tale fundamental groups and
profinite completions of topological fundamental groups~\cite[Expos\'{e} X, Corollaire 1.8]{SGAI}:
\[
\etpi{\mathbb{P}^1_\BAR{Q} \setminus \{0,1,\infty\},\overline{p}} \cong \etpi{\mathbb{P}^1_\mathbb{C} \setminus \{0,1,\infty\},p_\mathbb{C}}
\cong \widehat{\pi_1\left(\mathbb{P}^1_\mathbb{C} \setminus \{0,1,\infty\},p_\mathbb{C}\right)} \cong \widehat{\FreeTwo},\]
where the first two isomorphisms are canonical and the last given by our choice of generators
for $\pi_1(\mathbb{P}^1_\mathbb{C},p_\mathbb{C})$.
Furthermore, there is a homotopy exact sequence of \'{e}tale fundamental groups~\cite[Expos\'{e} IX, Th\'{e}or\`{e}me 6.1]{SGAI}:
\begin{equation}
\label{eq:HomotopySeq}
1 \to \etpi{\mathbb{P}^1_\BAR{Q} \setminus \{0,1,\infty\},\overline{p}} \to \etpi{\mathbb{P}^1_{\mathbb{Q}} \setminus \{0,1,\infty\},p} \to \GQ \to 1,
\end{equation}
which splits if $p$ is $\mathbb{Q}$-rational.
This induces an outer action
\begin{equation}\GQ \to \Out\left(\widehat{\FreeTwo}\right), \label{eq:OuterAction}\end{equation}
which is canonical~\cite[\S 3.2]{SchnepsGTIntro}.

The scheme $\mathbb{P}^1_\mathbb{Q} \setminus \{0,1,\infty\}$ can be replaced by any quasi-compact,
geometrically connected scheme $X$ over $\mathbb{Q}$ and $\mathbb{P}^1_\BAR{Q} \setminus \{0,1,\infty\}$
(resp. $\mathbb{P}^1_\mathbb{C} \setminus \{0,1,\infty\}$) by the base-change of $X$ to
$\BAR{Q}$ (resp. $\mathbb{C}$), but the choice of $\mathbb{P}^1_\mathbb{Q} \setminus \{0,1,\infty\}$ has special
properties, such as Theorem~\ref{thm:Belyi}, to be outlined in the next subsection.

\subsection{Belyi functions and dessins d'enfants}
A \emph{Belyi function} is a finite, \'{e}tale, connected cover of $\mathbb{P}^1_\BAR{Q} \setminus \{0,1,\infty\}$.
Due to~\cite[Expos\'{e} X, Corollaire 1.8]{SGAI},
we can equivalently view a Belyi function as a finite, \'{e}tale, connected cover of
$\mathbb{P}^1_{\mathbb{C}} \setminus \{0,1,\infty\}$,
which is a meromorphic function on a curve $X$ that is unbranched outside $\{0,1,\infty\}$.  A \emph{dessin d'enfant}
is a bipartite, connected graph $G$ with parts $V_0,V_1$ together with an embedding
$G \hookrightarrow X$ where $X$ is a compact, oriented, topological
2-manifold, whose image is the 1-skeleton of a CW-complex structure on $X$.

The following data are then equivalent~\cite{SchnepsDessinsSphere}:
\begin{enumerate}
\item an isomorphism class of Belyi functions of degree $n$;
\item an isomorphism class of dessin d'enfants with $n$ edges; and
\item a conjugacy class of transitive representations \[(\FreeTwo \cong) \, \pi_1(\mathbb{P}^1\setminus\{0,1,\infty\},\base) \rightarrow S_n.\]
\end{enumerate}
To a Belyi function $f$, we associate the dessin $f^{-1}([0,1])$ with
$V_0 = f^{-1}(0)$ and $V_1 = f^{-1}(1)$, and the
monodromy representation of $h: \FreeTwo \cong \pi_1(\mathbb{P}^1\setminus \{0,1,\infty\},\base) \rightarrow
S_n$.  It follows from the Riemann Existence Theorem that one can associate
a Belyi function to any dessin or transitive permutation representation $\FreeTwo \rightarrow S_n$.

There is a natural action of $\GQ$ on the category of Belyi functions: viewing
the category of Belyi functions as the category of \'{e}tale covers of $\mathbb{P}^1_{\BAR{Q}} \setminus \{0,1,\infty\}$
and given an automorphism $\sigma \in \GQ$, we can base-change by $\Spec \sigma$.  There
is an action of $\GQ$ on the category of representations of $\widehat{\FreeTwo}$ on finite
sets where $\sigma \in \GQ$ acts by sending $h$ to $h \circ \alpha(\sigma)$; the image of $h$ is defined
only up to isomorphism because $\GQ$ acts canonically only by outer automorphisms.
The category of Belyi functions is equivalent to the category of representations of $\FreeTwo$
on finite sets, (where $\FreeTwo$ is identified with $\pi_1\left(\mathbb{P}^1 \setminus \{0,1,\infty\},\base\right)$)
which is in turn equivalent to the category of representations of $\widehat{\FreeTwo}$ on finite sets
and therefore Equation~\ref{eq:OuterAction} yields a weak action of $\GQ$ on the category of Belyi functions.
The fact that the two actions are equivalent follows from the definition of the exact sequence
in Equation~\ref{eq:OuterAction}, and the fact that the group of isomorphism classes of self-equivalences
of the category of representations of $\widehat{\FreeTwo}$ on finite sets is canonically isomorphic
to $\Out\left(\widehat{\FreeTwo}\right)$.

A key result regarding the action of $\GQ$ follows from following theorem of Belyi.
\begin{thm}[\cite{BelyiOrig}, Theorem 4]
\label{thm:Belyi}
A curve admits a Belyi function if it is defined over $\overline{\mathbb{Q}}$.
\end{thm}
By considering the action of $\GQ$ on the $j$-invariants of smooth genus 1 curves over $\BAR{Q}$,
it follows the actions of $\GQ$ on $\widehat{\FreeTwo}$, the
category of Belyi functions, and the set of isomorphism classes of dessins are faithful~\cite{SchnepsDessinsSphere}.

\subsection{\texorpdfstring{$\GQ$-}{Galois }Invariants}

Properties of the action of $\GQ$ on $\widehat{\FreeTwo}$ (expressed as constraints on the image of $\GQ$ in the profinite Grothendieck-Teichm\"{u}ller group $\GT$) yield $\GQ$-invariants of dessins d'enfants.  Fix a Belyi function $f: X \rightarrow \mathbb{P}^1$ of degree $n$.  We obtain an associated
dessin $\Gamma \subseteq X$ and a monodromy representation
$h: \pi_1(\mathbb{P}^1\setminus \{0,1,\infty\},\base) \rightarrow S_n$.
Let $\psi_i \dashv n$ denote the cycle type of $\sigma_i = h(x_i)$ for $i \in \{0,1,\infty\}$.
It is evident
that the cycle type of the monodromy $(\psi_0,\psi_1,\psi_\infty)$
is $\GQ$-invariant.  In fact, $\psi_0$ is the degree multiset of $V_0$,
$\psi_1$ is the degree multiset of $V_1$, and $\psi_\infty$ is the multiset
of half the number of edges bounding each face of $\Gamma$~\cite[p.4]{GroThyIntro}.

Another $\GQ$-invariant is the \emph{monodromy group}, defined as the image of
the monodromy representation $h$, which is $\GQ$-invariant by definition
of the action of $\GQ$ on the category of Belyi functions.  A third invariant is the \emph{rational Nielsen class},
which associates a Belyi function $f$ of degree $n$ to the pair
\[\left(G \hookrightarrow S_n ,\left\{\left([\sigma_0^\lambda], [\sigma_1^\lambda], [\sigma_\infty^\lambda]\right) \mid \lambda \in
\hat{\mathbb{Z}}^\times\right\}\right),\]
where $G$ is the monodromy group of $f$ and $[u]$ denotes the conjugacy class of $u$ in $G$, which is defined up to simultaneous conjugation in $S_n$.  Let $\mathcal{N}(n)$ denote the set of rational Nielsen classes of degree of Belyi functions of degree $n$.

There are other combinatorial invariants, such as the Ellenberg's braid group invariant~\cite{Ellenberg}, Wood's Belyi-extending map invariant~\cite{Wood},
and Serre's lifting invariant~\cite[Section 3]{Ellenberg}.  Zapponi~\cite{Zapponi} defined an invariant for plane
trees (equivalently, Belyi polynomials) that is merely a sign $\pm 1$, but that
is particularly interesting in that it is not combinatorial.

\section{Statements of the main results}
\label{sec:Statements}

In Section~\ref{subsec:ThmsGaloisInvariant}, we describe
the basic properties of our $\GQ$-invariant of a certain family of dessins d'enfants.  In Section~\ref{subsec:newBelyiAndFaithful}, we describe a version of Belyi's Theorem and its consequences for the faithfulness of the action of $\GQ$ on Belyi functions whose monodromy cycle types above 0 and $\infty$ are equal.  In Section~\ref{subsec:NielsenWeak}, we give precise statements of our results that the rational Nielsen class and monodromy cycle type are coarse $\GQ$-invariants.  In Section~\ref{subsec:ToolsLowerBounds}, we describe the combinatorial framework we use to apply our $\GQ$-invariant to prove the results of Section~\ref{subsec:NielsenWeak}.

\subsection{A new \texorpdfstring{$\GQ$-}{Galois }invariant for Belyi functions with monodromy of cycle type \texorpdfstring{$(\psi,\mu,\psi)$}{psi,mu,psi}}
\label{subsec:ThmsGaloisInvariant}

In this subsection, we describe a new $\GQ$-invariant of a certain family of
dessins d'enfants.

\begin{defn}
Let $f$ be a Belyi function and let $n = \deg f$.
Suppose that $f$ has monodromy generators $\sigma_0,\sigma_1,\sigma_\infty$,
over $0,1,\infty,$ respectively.  The \emph{square-root class}
of $f$, denoted by $\Sqrt(f)$, is defined as
\[
\Sqrt(f) = \left\{\left(\sigma_0,\tau_1,\tau_1^{-1}\sigma_0^{-1}\right) \in S_n^3 \mid \tau_1^2 = \sigma_1
\text{ and } \sigma_\infty = \tau_1^{-1}\sigma_0\tau_1\right\}\]
modulo simultaneous conjugation in $S_n$.
\end{defn}
Because $\sigma_0,\sigma_1,\sigma_\infty$ are only defined up to simultaneous
conjugation in $S_n$, it makes sense to quotient by simultaneous conjugation.

\begin{rem}
If the monodromy cycle types of $f$ above $0$ and $\infty$ are different, then $\Sqrt(f) = \emptyset$ because $\sigma_0$ and $\sigma_\infty$ are not conjugate in $S_n$.  Even if the monodromy cycle types of $f$ above $0$ and $\infty$ are the same, it may still be the case that $\Sqrt(f) = \emptyset$.  Indeed, by Theorem~\ref{thm:Existence}, a result of Edmonds, Kulkarni, and Stong~\cite{EKS}, there exists a Belyi function $f$ with monodromy of cycle type 7 over $0,\infty$ and 421 over 1.  However, a permutation $\sigma_1$ of cycle type 421 cannot be a square in $S_7$.  Nevertheless, as we will see in the theorems later in this section, this invariant will be useful to us when it is non-trivial.
\end{rem}

\begin{defn}
Let the \emph{square-root cycle type class} of $f$, denoted
by $\Sqct(f)$, be the multiset of triples $(\psi_0,\psi_1,\psi_\infty)$
where $\psi_i$ is the cycle type of $\tau_i$,
for $(\tau_0,\tau_1,\tau_\infty) \in \Sqrt(f).$
\end{defn}
\begin{rem}
We let $\Sqct(f)$ be a multiset in order to ensure that $|\Sqrt(f)| = |\Sqct(f)|$.
\end{rem}

For each positive integer $n$, the action of $\GQ$
on the set of conjugacy classes of representations of $\FreeTwo$
in $S_n$ induces an action of $\GQ$ on the power set of the set
of such representations.  Hence, for all $\sigma \in \GQ$ and all
Belyi functions $f$, one can define $\Sqrt(f)^\sigma$.
A key property of the square-root class is its $\GQ$-equivariance.
This yields a key property of the square-root cycle type class,
which is that it is $\GQ$-invariant, and in certain cases it can distinguish $\GQ$-orbits
of dessins that are indistiguishable by the monodromy group and the rational
Nielsen class.  The square-root cycle type class is a purely combinatorial
invariant, albeit difficult to compute explicitly.  In order
to state the final properties of the square-root cycle type class, we
define the genus of an element of $\Sqct(f)$; for all $\psi = (\psi_0,\psi_1,\psi_\infty)$
with $\psi_i \dashv n$,
let
\[g\left(\psi\right)=
\frac{\sum_{i \in \{0,1,\infty\}} \ram(\psi_i)}{2}-n+1.\]
We can naturally extend $g$ to take arguments that are elements of $S_n$ instead.
If $\sigma_i$ is a permutation of cycle type $\psi_i$ for $i \in \{0,1,\infty\}$,
such that $\sigma_0\sigma_1\sigma_\infty = 1$ and the $\sigma_i$ generate a transitive subgroup of $S_n$,
the Riemann-Hurwitz formula implies that this is simply the genus of the curve $X$ that
admits a Belyi function with monodromy of cycle type $\left(\sigma_0,\sigma_1,\sigma_\infty\right).$

Now, we are prepared to state the key properties of the square root class and the square-root cycle type class.

\begin{thm}[Properties of $\Sqrt$ and $\Sqct$]
\label{thm:PropOfSqct}
The function $\Sqrt$ is $\GQ$-equivariant and thus the function $\Sqct$ is $\GQ$-invariant.
Let $f: X \rightarrow \mathbb{P}^1$ be a Belyi function and suppose that $X$ has genus $g$.  Then,
\begin{enumerate}[label=(\alph*)]
\item $|\Sqct(f)|$ is at most the number of non-trivial involutions
on $X$, and in particular, if $g > 1$, then $|\Sqct(f)| \le 84(g-1)-1$;
\item if there exist odd positive integers $k,c$ and
a triple $(\mu_0,\mu_1,\mu_\infty)\in \Sqct(f)$ such that $\mu_1$ has $c$ parts of size $k$
and no parts of size $2k$, then $\left|\Sqct(f)\right| = 1$;
\item if $g > 1$, then there exists
at most one triple $\psi = \left(\psi_0,\psi_1,\psi_\infty\right) \in \Sqct(f)$
such that $g\left(\psi\right) = 0$.
\end{enumerate}
\end{thm}

\subsection{Belyi functions of odd degree and the action of $\GQ$ on Belyi functions with monodromy of cycle type $(\phi,\mu,\phi)$}
\label{subsec:newBelyiAndFaithful}

In this section, we prove that $\GQ$ acts faithfully on the class of Belyi functions whose monodromy cycle types above 0 and $\infty$ are the same.  The proof relies on the properties of $\Sqrt$.  In particular, we prove the following theorem.

\begin{thm}
\label{thm:sqrtIsPowerful}
Let $\sigma \not= 1 \in \GQ$.  There exists a Belyi function $f$ that has odd degree such that $|\Sqrt(f)| = 1$ and $\Sqrt(f)^\sigma \not= \Sqrt(f)$.
\end{thm}

Recall that $\Sqrt(f) = \emptyset$ if the monodromy cycle types of $f$ above 0 and $\infty$ are different.  The following corollary is then immediate from  the $\GQ$-equivariance of $\Sqrt$ (Theorem~\ref{thm:PropOfSqct}).

\begin{cor}
\label{cor:GQFaithfulOn0InftySames}
The group $\GQ$ acts faithfully on the set of Belyi functions of odd degree whose monodromy cycle types above 0 and $\infty$ are the same.
\end{cor}

In the proof of Theorem~\ref{thm:sqrtIsPowerful}, we need to construct Belyi functions of odd degree.  To do so, we prove the following version of Belyi's Theorem.

\begin{thm}
\label{thm:OddDegreeBelyi}
Let $X$ be a curve that is defined over $\BAR{Q}$.  If $X$ admits a non-constant meromorphic function of odd degree that is defined over $\BAR{Q}$, then $X$ admits a Belyi function of odd degree.
\end{thm}

In particular, we apply Theorem~\ref{thm:OddDegreeBelyi} when $X$ is of genus 1 and use the action of $\GQ$ on $j$-invariants in the proof of Theorem~\ref{thm:sqrtIsPowerful}.

\subsection{The monodromy cycle type and the rational Nielsen class are imprecise invariants}
\label{subsec:NielsenWeak}
We use the $\GQ$-invariant of the previous subsection
to prove upper bounds on the precision of previously known
$\GQ$-invariants.

For all positive integers $n$, let
\[\Cl(N) = \max_{n \le N} \max_{\psi_1,\psi_2,\psi_3 \dashv n}
\left(
\begin{array}{c}
\text{number of }\GQ\text{-orbits of Belyi}\\
\text{functions with monodromy of}\\
\text{cycle type } (\psi_1,\psi_2,\psi_3)
\end{array}
\right).
\]
Using the tools of Section~\ref{subsec:ToolsLowerBounds}, we derive the following optimized lower bound.
\begin{thm}
\label{thm:ClLowerBound}
For all positive integers $N$, we have
\[\Cl(N) \ge \frac{1}{16} 2^{\sqrt{\frac{2N}{3}}}.\]
\end{thm}
For a positive integer $N$, let
\[\Cl'(N) = \max_{n \le N} \max_{c \in \mathcal{N}(n)}
\left(
\begin{array}{c}
\text{number of }\GQ\text{-orbits of Belyi}\\
\text{functions with rational Nielsen class } c
\end{array}
\right).
\]
We also prove the following theorem.
\begin{thm}
\label{thm:Cl'LowerBound}
For all $k < 2^{\sqrt{\frac{2}{3}}}$, we have
\[\Cl'(N) = \Omega\left(k^{\sqrt{N}}\right).\]
The monodromy groups of the rational Nielsen classes
achieving the given asymptotic inequality can be chosen to be $A_n$.
\end{thm}

\begin{rem}
Theorem~\ref{thm:ClLowerBound} is not special case of Theorem~\ref{thm:Cl'LowerBound},
because it provides an explicit constant as well as a base of $2^{\sqrt{\frac{2}{3}}}$
instead a base of arbitrarily close to $2^{\sqrt{\frac{2}{3}}}$.
\end{rem}

\subsection{Tools to prove the lower bounds}
\label{subsec:ToolsLowerBounds}
In this subsection, we state the specific consequences of Theorem~\ref{thm:PropOfSqct} that we use to prove the lower bounds
stated in Section~\ref{subsec:NielsenWeak}.  First, we describe a coarse analogue of $\Sqct$.

Let $n$ be a positive integer, and let $\psi, \mu \dashv n$.  We define a set $M(\psi,\mu)$,
of which $\Sqct(f)$ will be a subset for all Belyi functions $f$ of monodromy cycle type $(\psi, \mu, \psi)$.
First, we define an auxiliary set $M'(\psi,\mu)$. Suppose
that $\mu$ has $\ell_i$ parts of size $i$ for all $i$, and let $\psi_0 = n$.
\[M'(\psi, \mu) = \left\{
\begin{array}{l}
(u_0,u_1,\ldots,u_n) \mid \frac{\ell_i}{2} \le u_i \le \ell_i \text{ for }i \text{ and } i=0, u_i = \frac{\ell_i}{2}\\
\text{ for non-zero even }i, r+u_0+u_1+\cdots+u_n-n\\
\text{ is an even integer that is at most }2, \text{and there exists}\\
\text{an odd positive integer } c
\text{ such that } u_c = \ell_c \text{ is odd}\end{array}
\right\},\]
where $r$ is the number of parts of $\psi$.  Given a $(n+1)$-tuple
$u = (u_0,u_1,\ldots,u_n) \in M'(\psi,\mu)$, we associate partitions
$\alpha(u),\beta(u) \dashv n$.
The partition $\alpha(u)$ is defined by having $2u_0-\ell_0$ parts
of size 1 and $\ell_0-u_0$ parts of size 2, and
 $\beta(u)$ is defined by having $2u_k-\ell_k$ parts of size $k$ for $k$ odd,
and $\ell_{\frac{k}{2}}-u_{\frac{k}{2}} + 2u_{k} - \ell_k$ parts of size $k$
for $k$ even.  It is clear that $\alpha(u),\beta(u) \dashv n$.
Let $M(\psi,\mu) = \{(\psi,\beta(u),\alpha(u)) \mid u \in M'(\psi,\mu)\}$.
The constraints on $M'(\psi,\mu)$ are chosen so that elements of
$M(\psi,\mu)$ are \emph{consistent} in that the existence of a Belyi function
with monodromy cycle types given by any element of $M(\psi,\mu)$ would not
violate the Riemann-Hurwitz formula.

One specific application of part (b) of Theorem~\ref{thm:PropOfSqct},
is the following theorem.
\begin{thm}[Orbit-Splitting Theorem]
\label{thm:OrbitSplitting}
Let $n$ be a positive integer and let $\psi, \mu \dashv n$.
Then, there are at least $\left|M(\psi,\mu) \cap \mathcal{B}\right|$
$\GQ$-orbits of Belyi functions with monodromy of cycle type $(\psi,\mu,\psi)$.
\end{thm}
\begin{rem}
\label{rem:OrbitSplittingImpliesExistence}
In particular, we prove the existence of Belyi functions with monodromy of cycle type $(\psi,\mu,\psi)$ in the case when $\left|M(\psi,\mu) \cap \mathcal{B}\right| > 0$.
\end{rem}

An existence result for Belyi functions, due to Edmonds, Kulkarni, and Stong~\cite{EKS},
yields the following corollary.

\begin{cor}[$n$-cycle Orbit-Splitting Theorem]
\label{thm:nCycleOrbitSplitting}
If $\psi = n \dashv n$, then $M(\psi,\mu) \subseteq \mathcal{B}$.  Hence,
if $\mu \dashv n$, then there are at least $\left|M'(\psi,\mu)\right|$
$\GQ$-orbits of Belyi functions with monodromy of cycle type $(n,\mu,n)$.
\end{cor}

In certain cases, the constraint that $\mu_c = \ell_c$ is odd for some odd $c$
in the definition of $M'(\psi,\mu)$ is restrictive, in that
there are $\psi,\mu$ for which the Orbit-Splitting Theorem~\ref{thm:OrbitSplitting} gives weak bounds
on the number of $\GQ$-orbits of Belyi functions with monodromy
of cycle type $(\psi,\mu,\psi)$.  We prove an alternate form
that applies even in those cases, but is weaker in other cases.  For example,
consider $n = 11$, $\psi = 11$ and $\mu = 2222111$.
The $n$-cycle Orbit-Splitting Theorem implies that there are at least zero $\GQ$-orbits
of Belyi functions with monodromy of cycle type $(\psi,\mu,\psi)$;
the alternate form will imply that there are at least two $\GQ$-orbits.

Once again, let $n$ be a positive integer, and let $\psi, \mu \dashv n$.
Suppose that $\mu$ has $\ell_i$ parts of size $i$ for all $i$, and let $\psi_0 = n$.  Let
\[M'_0(\psi, \mu) = \left\{
\begin{array}{l}
(u_0,u_1,\ldots,u_n) \mid \frac{\ell_i}{2} \le u_i \le \ell_i \text{ for odd }i \text{ and } i=0,\\
\text{there exists an odd } i \text{ with } u_i \not= \frac{\ell_i}{2},
u_i = \frac{\ell_i}{2} \text{ for}\\
\text{non-zero even }i, \text{ and }r+u_0+u_1+\cdots+u_n = n + 2\\
\end{array}
\right\},\]
where $r$ is the number of parts of $\psi$.  Define
$M_0(\psi,\mu) = \{(\psi,\beta(u),\alpha(u)) \mid u \in M'_0(\psi,\mu)\}.$
We prove the following analogue of the Orbit-Splitting Theorem,
which follows from Theorem~\ref{thm:PropOfSqct}(c).

\begin{thm}[Orbit-Splitting Theorem, Alternate Form]
\label{thm:OrbitSplittingAlt}
Let $n$ be a positive integer and let $\psi, \mu \dashv n$.
Suppose that $\psi$ has $r$ parts and $\mu$ has $s$ parts,
and $2r+s < n$. Then, there are at least
$\left|M_0(\psi,\mu) \cap \mathcal{B}\right|$; and
$\GQ$-orbits of Belyi functions with monodromy of cycle type $(\psi,\mu,\psi)$.
\end{thm}
\begin{rem}
In particular, we prove the existence of Belyi functions with monodromy of cycle type $(\psi,\mu,\psi)$ in the case when $\left|M_0(\psi,\mu) \cap \mathcal{B}\right| > 0$.  (See also Remark~\ref{rem:OrbitSplittingImpliesExistence}.)
\end{rem}

Similar to the $n$-cycle Orbit-Splitting Theorem, we obtain the following corollary.

\begin{cor}[$n$-cycle Orbit-Splitting Theorem, Alternate Form]
\label{thm:nCycOrbitSplittingAlt}
If $\psi = n \dashv n$, then $M_0(\psi,\mu) \subseteq \mathcal{B}$. Hence,
if $\mu \dashv n$ has less than $n-2$ parts, then there are at least
$\left|M_0(\psi,\mu)\right|$ $\GQ$-orbits of Belyi functions with monodromy of cycle type $(n,\mu,n)$.
\end{cor}

\section{Proof of Theorem~\ref{thm:PropOfSqct}}
\label{sec:ProofProp}
Let $f$ and $t$ be affine coordinates centered at 0 on $\mathbb{P}^1_f$
and $\mathbb{P}^1_t$, respectively.  Define the morphism
$t = \frac{4f}{(f+1)^2} : \mathbb{P}^1_f \rightarrow \mathbb{P}^1_t$.  The proof of Theorem~\ref{thm:PropOfSqct} relies on the fact that $t$ is defined over $\mathbb{Q}$.
In Section~\ref{subsec:ComputeT_*}, we compute
the induced morphism $t_*$ of fundamental groups,
which we apply in Section~\ref{subsec:GQInvariance} to prove that
the square-root cycle type class is $\GQ$-invariant.
In Sections~\ref{subsec:PropOfSqctA} and~\ref{subsec:PropOfSqctC}, we use the geometric properties
of base-changing by $t$ to prove parts (a) and (c) of Theorem~\ref{thm:PropOfSqct}, respectively.  In Section~\ref{subsec:PropOfSqctB}, we use the combinatorial properties
of the monodromy representation of $t$ to prove Theorem~\ref{thm:PropOfSqct}(b).

\subsection{Computation of the morphism \texorpdfstring{$t_*$}{t\_*} of topological fundamental groups}
\label{subsec:ComputeT_*}
Notice that $t'(0) = 4$, and therefore $t$ preserves the topological tangential base-point $\base$.
The function $t$ defines a morphism from $\mathbb{P}^1 \setminus \{-1,0,1,\infty\}$ to $\mathbb{P}^1 \setminus \{0,1,\infty\}$.  We choose generators for
$\pi_1\left(\mathbb{P}^1 \setminus \{-1,0,1,\infty\},\base\right)$ as in Figure~\ref{fig:P1Minus4}.
\begin{figure}
\includegraphics[width=4in]{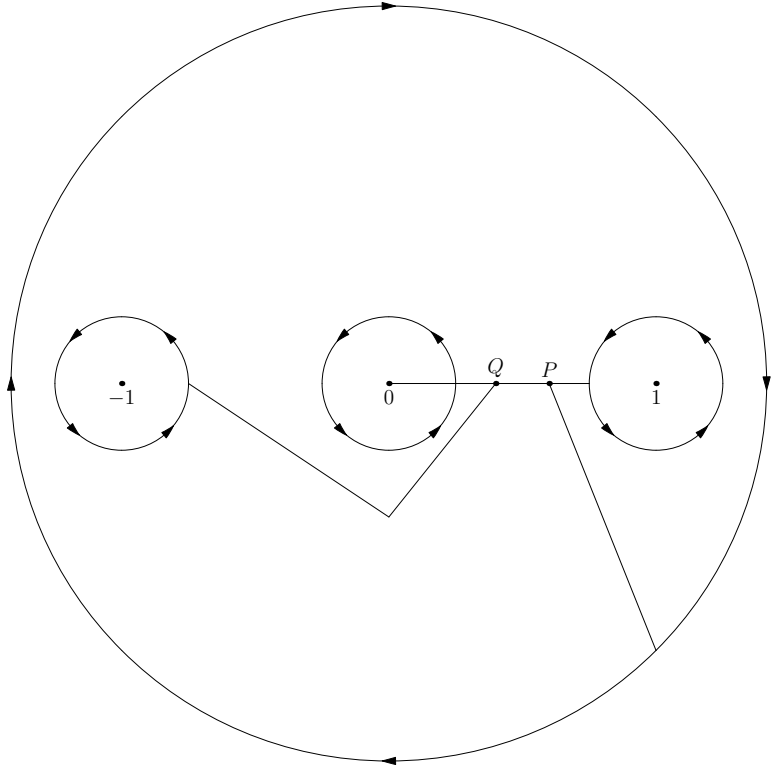}
\caption{{\bf Generators for $\pi_1\left(\mathbb{P}^1\setminus \{-1,0,1,\infty\},\protect\base\right).$} The homotopy classes $y_0,y_1,y_\infty$ are
defined similarly to Figure~\ref{fig:P1Minus3}.  The homotopy class $y_{-1}$ is defined by going from 0 to $Q$ along the segment, moving along the marked path to the circle around $-1$, traversing the circle around $-1$ counterclockwise, returning to $Q$ around the marked path, and then returning to 0 along the segment. It is evident that $y_{-1}y_0y_1y_\infty = 1$.} 
\label{fig:P1Minus4}
\end{figure}
It is clear that
\begin{align*}
t_*y_0 &= x_0\\
t_*y_1 &= x_1^2\\
t_*y_{-1} &= x_\infty^2\\ 
\end{align*}
where the equalities are up to base-point fixing homotopy.
Because $y_{-1}y_0y_1y_\infty = 1$ and $t_*$
is a homomorphism, we have
\[
t_*y_\infty = t_*\left(y_1^{-1}y_0^{-1}y_{-1}^{-1}\right)\\
=x_1^{-2} x_0^{-1} x_\infty^{-2} = x_1^{-2} x_0^{-1} x_0 x_1 x_0 x_1 = x_1^{-1}x_0x_1.\]

\subsection{Proof that \texorpdfstring{$\Sqrt$}{Sqrt} is \texorpdfstring{$\GQ$-}{Galois }equivariant}
\label{subsec:GQInvariance}

Let $n$ be a positive integer, and let $g: X \rightarrow \mathbb{P}^1_t$ be a Belyi function of degree $n$ defined on an algebraic curve $X$.
Let $X'$ be the normalization of $X \times_{\mathbb{P}^1_t} \mathbb{P}^1_f$,
and let $f: X' \rightarrow \mathbb{P}^1_f$ be the projection.
The curve $X'$ may not be irreducible.

\begin{defn}
We write $f = \Sigma(g)$, so that $\Sigma$ defines a function
from the set of isomorphism classes of Belyi functions to the
set of isomorphism classes of morphisms of curves $X' \rightarrow \mathbb{P}^1$,
where $X'$ is not necessarily irreducible.
\end{defn}

\[\begin{diagram}
X &\lTo^{\alpha} & X'\\
\dTo^g && \dTo_f \\
\mathbb{P}^1_t & \lTo_{t=\frac{4f}{(f+1)^2}} & \mathbb{P}^1_f
\end{diagram}\]

The projection $\alpha: X' \rightarrow X$ induces a bijection
between the fibers $g^{-1}(\base)$ and $(g')^{-1}(\base)$.  We order the fiber $g^{-1}(\base)$,
which gives an order on $(g')^{-1}(\base)$ via the restriction of $\alpha$.  Using these
orders, we can define the monodromy of $f$ and $g$ as fixed representations (not isomorphism
classes of representations) of $\pi_1(\mathbb{P}^1_f \setminus \{-1,0,1,\infty\},\base)$ and
$\pi_1(\mathbb{P}^1_t \setminus \{0,1,\infty\},\base)$ on $[n]$.  Let $p_k \in S_n$
be the image of $x_k$ under the representation of $\pi_1(\mathbb{P}^1_t \setminus \{0,1,\infty\},\base)$
for $k \in \{0,1,\infty\}$,
and let $\sigma_k$ be the image of $y_k$ under the monodromy
representation of $\pi_1(\mathbb{P}^1_f \setminus \{-1,0,1,\infty\},\base)$
for $k \in \{-1,0,1,\infty\}$.

For all Belyi functions $g$, the fact that $\Sigma(g)$ is \'{e}tale
outside $\{-1,0,1,\infty\}$ follows from the fact that \'{e}taleness is
preserved under base-change.  The following proposition is immediate
by lifting loops.

\begin{prop}
\label{prop:MonodromyProduct}
Let $g:X \to \mathbb{P}^1$ be a Belyi function, with monodromy generators $\tau_{0},\tau_1,\tau_\infty$.
Then, $\Sigma(g)$ is unbranched outside $\{-1,0,1,\infty\}$.  Let $\sigma_{-1},\sigma_0,\sigma_1,\sigma_\infty$
be the monodromy of the function $\Sigma(g)$ over $-1,0,1,\infty,$ respectively
(the permutations are defined up to simultaneous conjugation in $S_n$ because
we fixed loops of winding number 1 about each branch point in both $\mathbb{P}^1_t$
and $\mathbb{P}^1_f$).  Then, we have $\sigma_0 = \tau_0$, $\sigma_1 = \tau_1^2$, $\sigma_{-1} = \tau_\infty^2$,
and $\sigma_\infty = \tau_1^{-1}\tau_0\tau_1$.
\end{prop}

We are now ready to link the constructions of this subsection
to the square-root class.

\begin{defn}
Let $f: X \to \mathbb{P}^1$ be a Belyi function.  Define
\[\Sqrt'(f) = \{g \mid \Sigma(g) \cong f\},\]
and call $\Sqrt'(f)$ the \emph{fibered product square-root class of $f$}.
\end{defn}

\begin{thm}
\label{thm:SigmaSqrt}
\begin{enumerate}[label=(\alph*)]
\item The function $\Sqrt'$ is $\GQ$-equivariant.
\item Let $n$ be a positive integer, and
let $f: X \to \mathbb{P}^1$ be a Belyi function
of degree $n$.  Then, $\Sqrt(f)$ is
the set of monodromy triples of elements of $\Sqrt'(f)$.
\end{enumerate}
In particular, the function $\Sqrt$ is $\GQ$-equivariant.
\end{thm}
\begin{proof}
We begin by proving part (a).  We treat
Belyi functions as finite \'{e}tale covers of $\mathbb{P}^1_\BAR{Q} \setminus \{0,1,\infty\}$.
The fact that $\Sigma$ is $\GQ$-equivariant
then follows from the fact that base-changing by by $\Spec \sigma$ preserves fibered products and normalizations for all $\sigma \in \GQ$.

Part (b) follows immediately from Proposition~\ref{prop:MonodromyProduct}.
\end{proof}

\subsection{Proof of Part (a)}
\label{subsec:PropOfSqctA}

Let $f: X \to \mathbb{P}^1$ be a Belyi function.
The key to the proof of this part is to construct an injection
from $\Sqrt'(f)$ to the set of involutions on $X$.
The remainder of the statement follows from Hurwitz's Automorphism Theorem.

\begin{proof}[Proof of Theorem~\ref{thm:PropOfSqct}(a)]
Let $Inv(X)$ denote the set of non-trivial involutions on $X$.
We construct an injection $i: \Sqrt'(f) \to Inv(X)$.  Let
$g \in \Sqrt'(f)$, so that we have a diagram
\begin{equation}
\label{eq:PropSqcta}
\begin{diagram}
X &\lTo^{\alpha} & X'\\
\dTo^g && \dTo_f \\
\mathbb{P}^1_t & \lTo_{t=\frac{4f}{(f+1)^2}} & \mathbb{P}^1_f
\end{diagram}
\end{equation}
with $X$ the normalization of the fibered product $Y \times_{\mathbb{P}^1_t} \mathbb{P}^1_f$.
The bottom morphism is of degree 2, and the vertical morphisms
are of degree $n$, which implies that the top morphism $\alpha$
is also of degree 2.  There is an involution $\iota: X' \to X'$, which is
the unique deck transformation for the restriction of $\alpha$
to its unramified locus.  Let $i(g) = \iota$.  Note that
$\alpha: X' \to X$ is the quotient of $X'$ by $\iota$,
so that $\iota$ determines $\alpha$ up to composition
by an automorphism of $Y$.

To prove that $i$ is injective, it suffices to prove that $\alpha$, $f$,
and the bottom morphism uniquely determine $g$.  This is obvious,
because $\alpha$ is surjective and Diagram~\ref{eq:PropSqcta}
is required to commute.  Therefore, $|\Sqrt'(f)| \le |Inv(X)|$,
and the fact that $|\Sqrt(f)| \le |Inv(X)|$ follows by Theorem~\ref{thm:SigmaSqrt}(b).
\end{proof}

\subsection{Proof of Part (b)}
\label{subsec:PropOfSqctB}

We transfer to representations of $\FreeTwo$ to analyze
the fibered product square-root class.  Fix
a generating set $\FreeTwo = \langle x, y \rangle$ and a
positive integer $n$.
For all positive integers $k$, let $[k]$ denote the set $\{1,2,\ldots,k\}$.
Let $T_r$ be the set of conjugacy classes of transitive representations $m: \FreeTwo \rightarrow S_n$
such that there exists an odd positive integer $c$ such that $m(y)$
contains an odd number of cycles of length $c$ and no cycle of length $2c$.  Let $\xi$ denote the
representation of $\FreeTwo$ on $S_2$ with $\xi(x) = (1)(2)$ and $\xi(y) = (1 \, 2)$.

\begin{prop}
\label{prop:RepProduct}
Let $n$ be a positive integer. Let $m \in T_r$ be a permutation representation $m: \FreeTwo \rightarrow S_n$,
and let $m'$ be a permutation representation $m': \FreeTwo \rightarrow S_n$.
\begin{enumerate}[label=(\alph*)]
\item The product representation $m \times \xi$ is transitive.
\item If $m \times \xi \cong m' \times \xi$, then $m \cong m'$.
\end{enumerate}
\end{prop}
\begin{proof}
Suppose that $m,m'$ satisfy the conditions of the proposition.
Let $m_\xi = m \times \xi$ and let $m'_\xi = m \times \xi$.

First, we prove part (a).
Let $(a,b),(a',b') \in [n] \times [2]$, and we will prove that there exists
a word $w \in \FreeTwo$ such that $m_\xi(w)(a,b) = (a',b')$.
By definition, the permutation $m(y) \in S_n$
must have an odd cycle in its cycle decomposition.  Suppose that
$(p_1 \, p_2 \, \cdots \, p_k)$ be a cycle in $m(y)$ with $k$ odd.
Let $w_0,w_1 \in \FreeTwo$ be such that $m(w_0)(a) = p_1$
and $m(w_1)(p_1) = a'$; because $m$ is transitive, such $w_0,w_1$
exist.  If $\xi(w_1w_0)(b) = b'$, then we can take $w = w_1w_0$
because $m(w_1w_0)(a) = a'$.  Hence, we can assume that $\xi(w_1w_0) \not= b'$.
Let $w = w_1y^kw_0$.  Because $k$ is odd, $\xi(w)(b) = b'$, and it is
easy to see that $m(w)(a) = a'$.
It follows that $m'_\xi$ is transitive.

We now prove part (b).  Part (a) implies that $m'$ is also transitive.
There is an automorphism $\alpha$ of $[n] \times [2]$ such that
$\alpha \circ (m \times \xi)  = m' \times \xi$.
Let $G$ be the kernel of $\xi$; it is a normal subgroup of index 2 in $\FreeTwo$.
Note that the $m$-action (resp. $m'$-action) of $G$ on $[n] \times [2]$ fixes the second coordinate.
Because $m_\xi$ and $m'_\xi$ are transitive representations, the group
$\FreeTwo/G \cong (\mathbb{Z}/2\mathbb{Z})^+$ acts transitively
on the set of $m_\xi$-orbits (resp. $m'_\xi$-orbits) of $G$ in $[n] \times [2]$.
In particular, there are at most 2 $m_\xi$-orbits (resp. $m'_\xi$-orbits) of $G$.  Hence, the $m_\xi$-orbits (resp. $m'_\xi$-orbits) of $G$ are $[n] \times \{1\}$ and $[n] \times \{2\}$.
The action of $\alpha$ must preserve these orbits.
Therefore, the second coordinate of $\alpha(i,j)$ must either be $j$ for all $i,j$ or be $3-j$ for all $i,j$.
Furthermore, $G$ acts transitively on $[n] \times \{1\}$.

Suppose that $m(y)$ contains $2k+1$ of cycles of length $c$ and no cycle of length $2c$.
Then, $m_\xi(y)$ contains $2k+1$ cycles of length $2c$.  Therefore, $m'_\xi(y)$
must also contain $2k+1$ cycles of length $2c$.  Suppose that $m'(y)$ contains
$a$ cycles of length $c$ and $b$ cycles of length $2c$.  Then, $m'_\xi(y)$ contains
$a+2b$ cycles of length $2c$, from which it follows that $a$ is odd and thus
$a \ge 1$.  Let $\zeta'$ be a cycle of length $c$ in $m'_\xi(y)$, and let
$\tau' = \zeta' \times \xi(y)$ the corresponding cycle of length $2c$ in $m'_\xi(y)$.
Let $\tau = \alpha^{-1} \circ \tau' \circ \alpha$ be the corresponding
cycle of length $2c$ in $m_\xi$.  Because $m(y)$ does not contain any
cycle of length $2c$, we must have $\tau = \zeta \times \xi(y)$
for some cycle $\zeta$ of length $c$ in $m(y)$.

Without loss of generality, we assume that
$1$ is not fixed by $\zeta$, and we may
also assume that the second coordinate of $\alpha(1,1)$ is $1$.
Let $\alpha(1,1) = (\beta(1),1)$.
Then, we have $\tau^c(1,1) = (1,2)$ and $\tau^c(1,2) = (1,1)$, and similarly
that $\tau'^c(\beta(1),1) = (\beta(1),2)$ and $\tau'^c(\beta(1),2) = (\beta(1),1)$.

It suffices to prove that there is a permutation $\beta \in S_n$ such that
$\alpha(i,j) = (\beta(i),j)$, as this would imply that the representations $m$
and $m'$ differ only by conjugation by an element of $S_n$.  Fix $i$ and let $g \in G$ be such that
$m_\xi(g)(1,1) = (i,1)$.  We must have $m(g)(i) = 1$, from which it follows that $m_\xi(g)(1,2) = (i,2)$.
We have
\[m_\xi(g) \circ \tau^c (i,1) = (i,2).\]
It is clear that
\[m'_{\xi}(g)(\beta(i),1) = \alpha\circ m_\xi(g) \circ \alpha^{-1} = (\beta(1),1).\]
Thus, we have $m'(g)(\beta(i)) = \beta(1)$ from which it follows that $m'_{\xi}(\beta(i),2) = (\beta(1),2)$.
However, we have
\begin{align*}
\alpha(i,2) &= \alpha \circ m_\xi(g) \circ \tau^c \circ m_\xi(g)^{-1}(i,1)\\
&= \left(\alpha \circ m_\xi(g) \circ \alpha^{-1}\right) \circ \left(\alpha \circ \tau^c \circ \alpha^{-1}\right) (\alpha(1,1))\\
&= m'_{\xi}(g) \circ \tau'^c (\beta(1),1) = m'_{\xi}(g) (\beta(1),2) = (\beta(i),2),\end{align*}
as desired. The proposition follows.
\end{proof}

Fix the isomorphism
$\langle x,y \rangle = \FreeTwo \cong \pi_1(\mathbb{P}^1\setminus\{0,1,\infty\},\base)$
with $x \mapsto x_\infty$ and $y \mapsto x_1$.
Taking monodromy representations gives a bijection $K = K_n$ between
the set of isomorphism classes of degree $n$ Belyi functions
and the set of transitive representations $m: \FreeTwo \to S_n$.
An important auxiliary result that we use in the proof of Theorem~\ref{thm:PropOfSqct}
as well as the proof of the Orbit-Splitting Theorems is the following proposition.

\begin{prop}
\label{prop:RepFiberedProduct}
Fix a positive integer $n$. For all Belyi functions $g$ of degree $n$,
$K(\Sigma(g) \circ t) = K(g) \times \xi$, where $t = \frac{4f}{(f+1)^2}$.
\end{prop}

It is not necessary for the purposes of Proposition~\ref{prop:RepFiberedProduct} that $g$ has at most simple
ramification over 0.  We do not need $\Sigma(g)$ to
be a Belyi function, because we consider the composite $t \circ \Sigma(g)$.

\begin{proof}
Let $\mathcal{C}$ be the category of \'{e}tale covers of $\mathbb{P}^1 \setminus \{0,1,\infty\}$.
The function $K$ is the object function of a contravariant functor from $\mathcal{C}$
to ${\bf FinSet}^{\FreeTwo}$, the category of finite permutation representations of $\FreeTwo$.
It is well-known that $K$ is in fact an equivalence of categories.  In particular,
$K$ preserves products. But, $t \circ \Sigma(g) = g \times t$ (in $\mathcal{C}$), and the conclusion
follows.
\end{proof}

\begin{proof}[Proof of Theorem~\ref{thm:PropOfSqct}(b)]
Let $f: X \to \mathbb{P}^1$ be a Belyi function of odd degree $n$,
let $(\tau_0,\tau_1,\tau_\infty) \in \Sqrt(f)$,
and let $\mu \dashv n$ be the cycle type of $\tau_1$.
Suppose that $k,c$ are odd positive integers such that $\mu$ has $c$ parts of size $k$
and no parts of size $2k$.  Let $m$ be the representation
of $\FreeTwo$ on $S_n$ that sends $x$ to $\tau_0$ and $y$ to $\tau_1$.
By Proposition~\ref{prop:RepProduct}, if a representation $m': \FreeTwo \to S_n$
satisfies $m \times \xi \cong m' \times \xi$, then in fact $m \cong m'$.

Suppose that $\tau' = (\tau_0',\tau_1',\tau_\infty') \in \Sqrt(f)$
and $m': \FreeTwo \to S_n$ is the corresponding representation.
It follows from Theorem~\ref{thm:SigmaSqrt}(b) and Proposition~\ref{prop:RepFiberedProduct} that $m' \times \xi
\cong K(f \circ t) \cong m \times \xi$, which implies that
$m \cong m'$.  Therefore, $(\tau_0',\tau_1',\tau_\infty')$
is conjugate to $(\tau_0,\tau_1,\tau_\infty)$.  Since the choice
of $\tau'$ was arbitrary, we have $|\Sqrt(f)| = 1$ and the result follows.
\end{proof}

\subsection{Proof of Part (c)}
\label{subsec:PropOfSqctC}

Let $n$ be an odd positive integer, and let $\psi,\mu \dashv n$.
We use the fact that a hyperelliptic curve admits a unique involution
with a genus 0 quotient in the proof of Theorem~\ref{thm:PropOfSqct}(a).

\begin{proof}[Proof of Theorem~\ref{thm:PropOfSqct}(c)]
Let $T_0$ denote the set of isomorphism classes of Belyi functions whose domains are $\mathbb{P}^1$.  Note
that $g(\psi_0,\psi_1,\psi_\infty)$ is the genus of a curve
that admits a Belyi function
with monodromy cycle type $(\psi_0,\psi_1,\psi_\infty)$,
if such a curve exists.
Therefore, by Theorem~\ref{thm:SigmaSqrt}(b),
it suffices to prove that the restriction of $\Sigma$ to $T_0$ is injective.

Consider two commutative squares
\[\begin{array}{ccc}
\begin{diagram}
X &\lTo^{\alpha} & X'\\
\dTo^g && \dTo_f \\
\mathbb{P}^1_t & \lTo_{t=\frac{4f}{(f+1)^2}} & \mathbb{P}^1_f
\end{diagram}
& \text{ and }
& \begin{diagram}
X &\lTo^{\alpha'} & X'\\
\dTo^{g'} && \dTo_f \\
\mathbb{P}^1_t & \lTo_{t=\frac{4f}{(f+1)^2}} & \mathbb{P}^1_f
\end{diagram}
\end{array}\]
where in both diagrams $X$ is the normalization of the fibered product
$\mathbb{P}^1_z \times_{\mathbb{P}^1_t} \mathbb{P}^1_f$, and the left morphisms are Belyi functions
of degree $n$.
Because a hyperelliptic curve of genus at least 2 admits a unique degree 2 function
to $\mathbb{P}^1$, there must be an automorphism $\beta$ of the top left copy
of $\mathbb{P}^1$ such that $\alpha' = \beta \circ \alpha$.  Hence,
we have $g \circ \alpha = t \circ f$ and $(g' \circ \beta) \circ \alpha = t \circ f$.
Because $\alpha$ is surjective, it follows that $g = g' \circ \beta$.
\end{proof}

\section{Proofs of the Orbit-Splitting Theorems and the lower bounds on \texorpdfstring{$\Cl(n)$ and $\Cl'(n)$}{Cl(n) and Cl'(n)}}
\label{sec:ProofMainLower}

In Section~\ref{subsec:HurwitzExistence}, we review
a result that guarantees the existence of Belyi functions with particular prescribed monodromy cycle types.  In Section~\ref{subsec:OrbitSplProof}, we prove the Orbit-Splitting Theorems using Theorem~\ref{thm:PropOfSqct}.  In Section~\ref{subsec:PrimitiveSubgroups}, we review some group-theoretic preliminaries that we use in the proofs of Theorems~\ref{thm:ClLowerBound} and~\ref{thm:Cl'LowerBound}, which we give in Section~\ref{subsec:LowerBoundProofs}.

\subsection{Hurwitz existence problem}
\label{subsec:HurwitzExistence}

We investigate Belyi functions with monodromy of fixed cycle type.  Let
$\mathcal{B}$ be the set of monodromy cycle types of Belyi functions.
Determining $\mathcal{B}$ is an unsolved case of the Hurwitz existence problem, which
deals with the possible sequences of monodromy cycle types of \'{e}tale covers
of arbitrary curves over $\mathbb{C}$ with removed points, but is a purely
group-theoretic question regarding finite permutation representations of
the fundamental groups of Riemann surfaces with points removed.

In the case of $\mathbb{P}^1 \setminus \{0,1,\infty\}$, the question is:
given a finite group $G$ and conjugacy
classes $c_0,c_1,c_\infty$, how many triples $(\sigma_0,\sigma_1,\sigma_\infty)$
are there of elements $\sigma_i \in c_i$ such that $\sigma_0\sigma_1\sigma_\infty = 1$?
When the finite group $G$ is replaced by a general linear group $GL_n(\mathbb{C})$, the analogous group-theoretic question is called the \emph{Deligne-Simpson problem}.
There is a formula for the number of solutions in terms of the characters of $G$ (see, for example,
Serre \cite[Theorem 7.2.1]{SerreGalois}), but this is not simple to evaluate in general.
Edmonds, Kulkarni, and Stong~\cite{EKS} construct a family of elements of $\mathcal{B}$.

\begin{thm}[\cite{EKS}, Proposition 5.2]
\label{thm:Existence}
Let $n$ be a positive integer, and let $\alpha, \beta \dashv n$.
Let $P$ be the total number of parts of $\alpha, \beta$.
A Belyi function with monodromy of cycle type $(\alpha,\beta,n)$ exists if and
only if $P \equiv n+1 \pmod{2}$ and $P \le n + 1$.
\end{thm}

Necessity follows immediately from the Riemann-Hurwitz formula, and sufficiency
is proven constructively.  If one of the partitions is not $n$, the Riemann-Hurwitz condition
on the total number of parts of the three partitions is not in general sufficient.

\subsection{Proofs of the Orbit-Splitting Theorems}
\label{subsec:OrbitSplProof}
Fix a integer $n$ and partitions $\psi,\mu \dashv n$.  For
$\alpha, \beta \dashv n$,
let $S_{\alpha,\beta}$ be the set of isomorphism
classes of Belyi functions with monodromy of cycle type
$(\psi,\beta,\alpha)$.  Let
\[S = \bigcup_{(\psi,\beta,\alpha) \in M(\psi,\mu) \cap \mathcal{B}} S_{\alpha,\beta}
\qquad \text{and} \qquad
S_0 = \bigcup_{(\psi,\beta,\alpha) \in M_0(\psi,\mu) \cap \mathcal{B}} S_{\alpha,\beta}.\]
Let $f \in \Sigma(S) \cup \Sigma(S_0)$.
Proposition~\ref{prop:MonodromyProduct} implies that $f$ is
unbranched outside $\{0,1,\infty\}$
and has monodromy of cycle type $(\psi,\mu,\psi)$.
By Propositions~\ref{prop:RepProduct}(a) and~\ref{prop:RepFiberedProduct}, the monodromy
of $f$ acts transitively on the fiber above the base-point,
and it follows that the domain of $f$ is irreducible and that $f$ is a Belyi function.
The Orbit-Splitting Theorem, in its ordinary and alternate forms, follow from Theorem~\ref{thm:PropOfSqct}
parts (b) and (c), respectively.

\begin{proof}[Proof of the Orbit-Splitting Theorem]
By Theorem~\ref{thm:PropOfSqct}(b), we have $|\Sqct(f)|=1$
for all $f \in \Sigma(S)$.  By construction, $\Sqct(f)$
can take any value in $M(\psi,\mu) \cap \mathcal{B}$
as $f$ ranges over $S$.  Because $\Sqct$ is $\GQ$-invariant,
the theorem follows.
\end{proof}

\begin{proof}[Proof of the Orbit-Splitting Theorem, Alternate Form]
By Theorem~\ref{thm:PropOfSqct}(b) and the
construction of $S_0$, $\Sqct(f)$ contains
exactly one element $(\psi_0,\psi_1,\psi_\infty)$ such that
$g(\psi_0,\psi_1,\psi_\infty) = 0$ for all $f \in \Sigma(S_0)$.
Denote this element by $R(f)$.  Because $\Sqct$ is $\GQ$-invariant,
so is $R(f)$.  By construction of $S_0$, $R(f)$ can take all values
in $M'(\psi,\mu) \cap \mathcal{B}$ as $f$ ranges over $S_0$,
and the theorem follows.
\end{proof}

By construction, the assertion that $M(\psi,\mu) \subseteq \mathcal{B}$
would not violate the Riemann-Hurwitz formula.  The fact that $M(\psi,\mu) \subseteq \mathcal{B}$
when $\psi = n \dashv n$ is immediate by Theorem~\ref{thm:Existence}, and the $n$-cycle Orbit-Splitting Theorems follow.

\subsection{Primitive subgroups of \texorpdfstring{$S_n$}{S\_n}}
\label{subsec:PrimitiveSubgroups}

In order to control the monodromy groups of the Belyi functions that we consider, we need a result on primitive subgroups of $S_n$, from
Dixon-Mortimer \cite{DixonMortimer} but due to Jordan.  We also
need a result describing permutation groups that contain short length cycles.

\begin{thm}[\cite{DixonMortimer}, Example 3.3.1]
\label{thm:PrimitiveDegree4}
Let $n \ge 9$, let $G$ be a subgroup of $S_n$,
and suppose that there exists a nonidentity $\sigma \in G$
with at least $n-4$ fixed points.  If $G$ does
not contain a transposition or a 3-cycle, then $G$ is not primitive.
\end{thm}

\begin{thm}[\cite{DixonMortimer}, Theorem 3.3E]
\label{thm:Primitive3Cycle}
Let $q$ be a prime, and let $n > q+2$.  If a primitive subgroup $G$
of $S_n$ contains a $q$-cycle, then $G$ contains $A_n$.
\end{thm}

The form that we will need is the following proposition, which is immediate
from Theorems~\ref{thm:PrimitiveDegree4} and~\ref{thm:Primitive3Cycle}.

\begin{prop}
\label{prop:p,22}
Let $p > 7$ be a prime, and let $G$ be a subgroup of $S_p$
that contains a $p$-cycle and a double transposition.  Then
$G$ contains $A_p$.
\end{prop}
\begin{proof}
A subgroup of $S_p$ that contains a $p$-cycle is primitive, and
a double transposition in $S_p$ has $p-4$ fixed points.  By Theorem~\ref{thm:PrimitiveDegree4},
$G$ contains a 2-cycle or a 3-cycle.  In both cases, Theorem~\ref{thm:Primitive3Cycle}
implies that $G$ contains $A_p$, as claimed.
\end{proof}

\begin{rem}[Noam Elkies, private communication]
The proposition is false for $p = 5, 7$.
For $p = 5$, one can take $G = D_{10}$, and for $p = 7$, one can take $G = PGL_3(\mathbb{Z}/2\mathbb{Z})$.
\end{rem}

\subsection{Proofs of Theorems~\ref{thm:ClLowerBound} and~\ref{thm:Cl'LowerBound}}
\label{subsec:LowerBoundProofs}

We derive Theorems~\ref{thm:ClLowerBound} and~\ref{thm:Cl'LowerBound} from the Orbit-Splitting Theorem and
the results quoted in the preceding section.  First, we begin with a
few computational lemmata, whose proofs are deferred to Appendix~\ref{app:CompLemmata}.

For positive integers $t$ and $k$ with $k \le t$, let
\[f_t(k) = \left\lfloor \frac{4t+2}{2k-1}\right\rfloor.\]
For a positive integer $t$, let
\[n_0(t) = 2t+1 + \sum_{i=1}^{t} 2(2k-1)(f_t(k)-1).\]

\begin{lemma}
\label{lemma:BoundN_0(t)}
For all positive integers $t$, we have
\[4t^2 + 12t + 1< n_0(t) < 6(t+1)^2 - 4.\]
\end{lemma}

\begin{lemma}
\label{lemma:BoundSumf}
Let $t$ be a positive integer.  Then, we have
\[\sum_{k=1}^t (f_t(k)-1) \le \frac{n_0(t)}{4}.\]
\end{lemma}

\begin{lemma}
\label{lemma:BoundProdf}
Let $t$ be a positive integer. Then, we have
\[\prod_{k=1}^t f_t(k) > 2^{2t}.\]
\end{lemma}

\begin{proof}[Proof of Theorem~\ref{thm:ClLowerBound}]
Fix a positive integer $t$, and let $n = n_0(t)$.
We prove a lower bound on $\Cl(n)$ that will imply
the theorem.  Let $\psi = n \dashv n$. Define the partition $\mu \dashv n$ to have $2f(k)-2$ parts
of size $2k-1$ for $1 \le k \le t$ and 1 part of size $2t + 1$.

We claim that
\begin{equation}
\label{eq:M'f}
\left|M'(\psi,\mu)\right| \ge \prod_{k=1}^t f_t(k).
\end{equation}
Let $S$ be the set of tuples $(v_0,v_1,\ldots,v_n)$ such that
$f_t(k)-1 \le v_{2k-1} \le 2f_t(k)-2$ for all $1 \le k \le t$, $v_{2t+1} = 1$,
$v_i = 0$ for all $i > 2t+1$ and $i = 2,4,\ldots,2t$, and
$v_0 = n-r(v)$, where  \[r(v) = \sum_{k=1}^t (2f_t(k)-2-v_k).\]
Notice that $v_{2t+1} = 1$, and $\mu$ has 1 part of size $2t+1$
and no parts of size $4t+2$.
Hence, to prove Equation~\ref{eq:M'f}, it suffices to prove that
$S \subseteq M'(\psi,\mu)$.
It suffices to prove that $r(v) \le \frac{n}{2}$.  Indeed, we have
\[\frac{r(v)}{2} \le \sum_{k=1}^t (f_t(k)-1).\]
Lemma~\ref{lemma:BoundSumf} implies that $r(v) \le \frac{n}{2}$ for all $t,v$.

The $n$-cycle Orbit-Splitting Theorem~\ref{thm:nCycleOrbitSplitting} implies that $\Cl(n) \ge |M(\psi,\mu)| \ge \prod_{k=1}^u f(k)$.
By Lemma~\ref{lemma:BoundProdf}, it follows that $\Cl(n) \ge 2^{2t}$, and
Lemma~\ref{lemma:BoundN_0(t)} yields that
\[\Cl(6(t+1)^2) \ge 2^{2t}.\]

We now let $t$ vary. Let $N \ge 24$ be a positive integer.  If $6(t+1)^2 \le N < 6(t+2)^2$,
then we have
\[\log_2\Cl(N) \ge 2t > 2\left(\sqrt{\frac{N}{6}}-2\right) = \sqrt{\frac{2N}{3}} - 4.\]
It follows that
\[\Cl(N) \ge \frac{1}{16} 2^{\sqrt{\frac{2N}{3}}}.\]
The bound is trivial for $N < 24$, and thus we have established the result for all $N$.
\end{proof}

\begin{rem}
A simpler construction can establish that $\Cl(N) = \Omega\left(2^{\sqrt{\frac{N}{2}}}\right)$.
\end{rem}

\begin{proof}[Proof of Theorem~\ref{thm:Cl'LowerBound}]
As in the previous proof, let $t$ be a positive integer.
Let \[n_1(t) = 4 + 2t+1 + \sum_{i=1}^{t} 2(2k-1)\left\lfloor \frac{4t+2}{2k-1}-1 \right\rfloor.\]
Let $n(t)$ be the smallest prime number that is at least $n_1(t)$.  Let $\epsilon(t) = \frac{n(t)}{n_0(t)} - 1$.

Fix $t$, and let $n = n(t)$.  It is clear that $n_1(t) > 2$, which implies that $n \equiv n_1(t) \pmod{2}$.
Let $2\alpha+1 = 2t+1 + n - n_0$.  Let $\psi = (n) \dashv n$, and let
$\mu \dashv n$ be the partition of $n$ with $f(k)$ parts of size $2k- 1$
for $1 \le k \le t$, two parts of size 2, and one part of size $n-n_0(t)$.
By Lemma~\ref{lemma:BoundN_0(t)}, we have $n_1(t) \le n_0(t) + 4 < 6(t+1)^2$,
which implies that $n < 6(t+1)^2(1+\epsilon(t))$.

We claim that
\begin{equation}
\label{eq:M'f2}
\left|M(\psi,\mu) \cap \mathcal{B}\right| \ge \prod_{k=1}^t f(k).
\end{equation}
Let $S$ be the set of tuples $(v_0,v_1,\ldots,v_n)$ such that
$f(k)-1 \le v_{2k+1} \le 2f_t(k)-2$ for all $1 \le k \le t$,
$v_{n-n_0(t)} = 1$, $v_0 = n-r(v)$ where
\[r(v) = 1 + \sum_{k=1}^t (2f_t(k)-2-v_k),\]
and $v_i = 0$ for all other $i$.
It follows from Lemma~\ref{lemma:BoundSumf} that $r(v) \le \frac{n}{2}$
for all $v,t$, which implies that $S \subseteq M'(\psi,\mu)$.
Notice that $v_{n-n_0(t)} = 1$, and $\mu$ has 1 part of size $n-n_0(t)$
and no parts of size $2n-2n_0(t)$.
Equation~\ref{eq:M'f2} follows.

Let $f$ be a Belyi function with monodromy of cycle type $(\psi,\mu,\psi)$ and monodromy generators
$\sigma_0,\sigma_1,\sigma_\infty$ over $0,1,\infty$, respectively.
By definition, the permutation
$\sigma_1^{(2t-1)!!}$ is a double transposition.  Because
\[n \ge n_1(t) = n_0(t) + 4 \ge n_0(1) + 4 = 9,\]
Proposition~\ref{prop:p,22} implies that the monodromy group $G$, which is generated by $\sigma_0$ and $\sigma_1$,
contains $A_n$.  The fact that $\sigma_0$ and $\sigma_1$ are even implies
that $G = A_n$.  There are two conjugacy classes of $n$-cycles in $A_n$,
so that $\sigma_0$ and $\sigma_\infty$ can lie in the same conjugacy class
or in different conjugacy classes.  Because $\sigma_0$ and $\sigma_1$ are only
defined up to conjugation in $S_n$, the case of both monodromy generators being in one
conjugacy class lies in the same rational Nielsen class as the case of both monodromy
generators being in the other rational Nielsen class.  Furthermore,
the $S_n$-conjugacy class of permutations of cycle type $\psi$ forms
a single $A_n$-conjugacy class.  Thus, there are at most two possible rational
Nielsen classes of Belyi functions with monodromy of cycle type $(\psi,\mu,\psi)$.

By the $n$-cycle Orbit-Splitting Theorem~\ref{thm:nCycleOrbitSplitting}, there are at least
$|M(\psi,\mu)|\ge \prod_{k=1}^t f(k)$ Belyi functions with monodromy of cycle type $(\psi,\mu,\psi)$.
The previous paragraph and Lemma~\ref{lemma:BoundProdf} then yield that
\[\Cl'(6(t+1)^2(1+\epsilon(t))) \ge \frac{1}{2} \prod_{k=1}^t f(k) > 2^{2t-1}\]
for all positive integers $t$.

We now let $t$ vary.  It follows from Lemma~\ref{lemma:BoundN_0(t)} that $\lim_{t \to \infty} \frac{n_1(t)}{n_0(t)} = 1$.  Because
\[ \lim_{t \rightarrow \infty} n_0(t) = \infty,\] the Prime Number Theorem
implies that
\[\lim_{t \rightarrow \infty} (1 + \epsilon(t)) = \lim_{t \rightarrow \infty} \frac{n(t)}{n_0(t)} = \lim_{t \to \infty} \frac{n(t)}{n_0(t)} = 1.\]  Fix a constant $k < 2^{\sqrt{\frac{2}{3}}}$.  Let $T$
be a positive integer such that
\[1 + \epsilon(t) < \frac{2}{3 \left(\log_2 k\right)^2}\]
for all $t > T$; such a $T$ exists because $\lim_{t \rightarrow \infty} \epsilon(t) = 0$.
Let $P = n_0(T)(1 + \epsilon(T))$, and let $N \ge P$.  There exist an integer $t \ge T$
such that \[n_0(t+1)(1 + \epsilon(t+1)) \le N < n_0(t+2)(1 + \epsilon(t+2)).\]  Then,
by Lemma~\ref{lemma:BoundSumf}, we have that $N < 6(t+2)^2(1 + \epsilon(t+2)).$
It follows that
\[t > \sqrt{\frac{N}{6(1 + \epsilon(t+2))}}-2.\]
The fact that $\Cl'$ is non-decreasing implies that
\[\log_2\Cl'(N) \ge 2t-1 >
\sqrt{\frac{2N}{3(1 + \epsilon(t+2))}}-5 > \sqrt{N}\log_2k - 5.\]
The theorem follows.
\end{proof}

\section{Proofs of Theorems~\ref{thm:sqrtIsPowerful} and~\ref{thm:OddDegreeBelyi}}
\label{sec:FaithfullnessProofs}

In Section~\ref{subsec:oddBelyiProof}, we prove Theorem~\ref{thm:OddDegreeBelyi}, a variant of Belyi's Theorem for Belyi functions of odd degree.  In Section~\ref{subsec:sqrtIsPowerfulProof}, we apply Theorems~\ref{thm:PropOfSqct} and~\ref{thm:OddDegreeBelyi} to prove Theorem~\ref{thm:sqrtIsPowerful}.

\subsection{Proof of Theorem~\ref{thm:OddDegreeBelyi}}
\label{subsec:oddBelyiProof}

For a morphism $f$, let $B(f)$ denote the branch locus of $f$.
We will adapt Belyi's first proof~\cite{BelyiOrig} to the setting of Belyi functions of odd degree.  We start with an arbitrary $\BAR{Q}$-morphism $f$ that has odd degree from an algebraic curve $X$ that is defined over $\BAR{Q}$ to $\mathbb{P}^1$.  Up to an automorphism of $\mathbb{P}^1$, we have $B(f) \subseteq \mathbb{P}^1(\BAR{Q})$.  We then successively compose $f$ with odd-degree polynomials until the branch locus of the composite is contained in $\mathbb{P}^1(\mathbb{Q})$.  We finish by composing with odd-degree rational functions to force the branch locus of the composite to lie within $\{0,1,\infty\}$. Our specific choice of polynomials and rational functions differs from Belyi's original choices because we restrict ourselves to functions that have odd degree.

Theorem~\ref{thm:OddDegreeBelyi} will follow quite simply from the following proposition.

\begin{prop}
\label{prop:collapseUsingAnOddDegree}
Let $S \subseteq \mathbb{P}^1_\BAR{Q}$.  Then, there exists a non-constant morphism $f: \mathbb{P}^1 \to \mathbb{P}^1$ that is defined over $\mathbb{Q}$ such that
\begin{enumerate}[label=(\arabic*)]
\item $f(S) \cup B(f) \subseteq \{0,1,\infty\}$; and
\item $f$ has odd degree.
\end{enumerate}
\end{prop}

We collapse the branch locus into $\mathbb{P}^1(\mathbb{Q})$ using repeated applications of the following lemma.

\begin{lemma}
\label{lem:collapseTowardQUsingOddDegree}
Let $S \subseteq \BAR{Q} \setminus \mathbb{Q}$ be a finite, non-empty set that is stable under the action of $\GQ$.  Then, there exists a non-constant polynomial $f \in \mathbb{Q}[x]$ such that
\begin{enumerate}[label=(\arabic*)]
\item $f(S) = \{0\}$;
\item $\left|B(f) \setminus \mathbb{P}^1(\mathbb{Q})\right| < |S|$; and
\item $f$ has odd degree.
\end{enumerate}
\end{lemma}
\begin{proof}
We do casework on the parity of $|S|$ to define $f$.
\begin{casework}
\item $|S|$ is odd.  We can follow Belyi~\cite{BelyiOrig} and let
\[f(x) = \prod_{s \in S} (x-s).\]
Because $B(f) = \{\infty\} \cup f(V(f'))$ and $\deg f' = |S|-1$, we are done.
\item $|S|$ is even.  Let
\[h(x) = \prod_{s \in S} (x-s).\]
Let $\beta \in \mathbb{Q}$ be such that $h'(\beta) \not= 0$.  Let $\alpha$ be the solution to the linear equation
\[(\beta - \alpha) h'(\beta) + h(\beta) = 0.\]
It is evident that $\alpha \in \mathbb{Q}$.  Let
\[f(x) = h(x)(x-\alpha).\]
Note that by construction, we have $f'(\beta) = 0$, and hence $f(\beta)$ is a rational branch point of $f$.  It follows that that \[|B(f) \setminus \mathbb{P}^1(\mathbb{Q})| \le \deg f' - 1 = |S| - 1,\] which completes the proof.
\end{casework}
\end{proof}

We then collapse $S$ to 3 points when $S \subseteq \mathbb{P}^1(\mathbb{Q})$ using the following lemma repeatedly.

\begin{lemma}
\label{lem:collapseASubsetOfQUsingOddDegree}
Given $r \in \mathbb{Q} \setminus \{0,1\}$, there exists a non-constant morphism $f : \mathbb{P}^1 \to \mathbb{P}^1$ that is defined over $\mathbb{Q}$ such that
\begin{enumerate}[label=(\arabic*)]
\item $f(\{0,1,\infty,r\}) \cup B(f) \subseteq \{0,1,\infty\}$; and
\item $f$ has odd degree.
\end{enumerate}
\end{lemma}
\begin{proof}
By applying an automorphism of $\mathbb{P}^1 \setminus \{0,1,\infty\}$, we can and will assume that $r \in (0,1)$ for the remainder of this proof.

Write $r = \frac{p}{q}$ with $(p,q) = 1$ and $0 < p < q$.  We do casework on the 2-adic valuation of $q$.  Each case will depend on the previously proven cases.
\begin{casework}
\item $q$ is odd.  Following Belyi~\cite{BelyiOrig}, we let
\[f(x) = \frac{q^q}{p^p(q-p)^{q-p}} x^p (1-x)^{q-p}.\]
We have $B(f) = \{0,1,\infty\}$ as well as $f(0) = f(1) = 0$ and $f(r) = 1$.  Because $q$ is odd, $f$ has odd degree.
\item $q$ is even.  We will need to divide into subcases based on the residue of $q$ modulo 4 later.  Firstly, let
\[g(x) = \frac{q}{q-p} x^2 \left(x - r\right),\]
and let $h = \frac{g}{g-1}$.  Note that $g(1) = 1$.  The logarithmic derivative of $g$ is
\[\frac{g'}{g} = \frac{2}{x} + \frac{1}{x-r} = \frac{3x - 2r}{x(x-r)}.\]
Therefore, the only critical point of $g$ that is not a zero or a pole is $\frac{2p}{3q}$, and
\[g\left(\frac{2p}{3q}\right) = -\frac{4p^3}{27q^2(q-p)}.\]
Therefore, we have
\[h\left(\frac{2p}{3q}\right) = \frac{4p^3}{4p^3 + 27q^2(q-p)} = \frac{p^3}{p^3 + 27\left(\frac{q}{2}\right)^2(q-p)},\]
a rational number that we denote by $r_2$.
Let us now analyze the function $h$.  We have
$B(h) = \{0,r_2,1,\infty\},$
as well as $h(0) = h(r) = 0$, $h(\infty) = 1$, and $h(1) = \infty$.  We now need to divide into cases based on whether $q$ is divisible by 4.
\begin{subcasework}
\item $q$ is divisible by 4.  Then, note that $r_2 \in (0,1)$
is a fraction with odd denominator.  By Case 1, we can find a function $f_0$ such that $f_0$ has odd degree, $f_0(\{0,1,\infty,r_2\}) \subseteq \{0,1,\infty\}$ and $B(f_0) \subseteq \{0,1,\infty\}$.  Let $f = f_0 \circ h$.  It is evident that $f$ has the desired properties.
\item $q$ is not divisible by 4.  Then, we have $q \equiv 2 \pmod{4}$, from which it follows that $p \equiv q-p \pmod{4}$.  Hence, we have
\[p^3 + 27\left(\frac{q}{2}\right)^2(q-p) \equiv p\left(p^2 + 27\left(\frac{q}{2}\right)^2\right) \equiv p \cdot 28 \equiv 0 \pmod{4},\]
so that $r_2 \in (0,1)$ is a rational number with odd numerator and a denominator that is divisible by 4.
By Subcase 2.1, we can find a function $f_0$ with the properties asserted in the lemma for $r = r_2$.  We can then proceed as in Subcase 2.1, and let $f = f_0 \circ h$.
\end{subcasework}
\end{casework}
\end{proof}

We are now ready to complete the proofs of Proposition~\ref{prop:collapseUsingAnOddDegree} and Theorem~\ref{thm:OddDegreeBelyi}, following Belyi~\cite{BelyiOrig}.

\begin{proof}[Proof of Proposition~\ref{prop:collapseUsingAnOddDegree}]
By enlarging $S$, we can assume that $S$ is $\GQ$-stable.  Let $T_0 = S \setminus \mathbb{P}^1(\mathbb{Q})$.  We claim that there exists a polynomial $h \in \mathbb{Q}[x]$ such that $h(T_0) = \{0\}$ and $B(h) \subseteq \mathbb{Q}$.  To see this, we can apply Lemma~\ref{lem:collapseTowardQUsingOddDegree} repeatedly.  Indeed, if $|T_i| > 0$, let $h_i$ be the polynomial constructed by Lemma~\ref{lem:collapseTowardQUsingOddDegree} for the set $T_i$, and let $T_{i+1} = B(h_i) \setminus \mathbb{Q}$.  Because $h_i$ is defined over $\mathbb{Q}$, the set $T_{i+1}$ is $\GQ$-stable.  The process terminates after a finite number of steps because $|T_{i+1}| < |T_i|$ for all $i$.  Suppose that $N$ steps are required.  Then, let $h = h_{N-1} \circ \cdots \circ h_0$.  It is evident that $h$ has the required properties.

Let $U_0 = h(S) \cup B(h)$.  We will find a rational function $g$ that is defined over $\mathbb{Q}$ and has odd degree such that $g(U_0) \subseteq \{0,1,\infty\}$ and $B(g) \subseteq \{0,1,\infty\}$.  By construction, $U_0$ is a subset of $\mathbb{Q}$. If $|U_0| \le 3$, then we can simply take $g$ to be an appropriate automorphism of $\mathbb{P}^1$.

Hence, we may assume that $|U_0| \ge 4$.  We can find an automorphism $\theta$ of $\mathbb{P}^1$ that is defined over $\mathbb{Q}$ and such that $\theta(U) \supseteq \{0,1,\infty\}$.  We will apply Lemma~\ref{lem:collapseTowardQUsingOddDegree} repeatedly to conclude the proof.  If $|U_i| \ge 3$, let $\alpha \in U_i \setminus \{0,1,\infty\}$ and let $g_i$ be the rational function constructed by Lemma~\ref{lem:collapseASubsetOfQUsingOddDegree} for $r = \alpha$.  Then, let $U_{i+1} = g_i(U_i)$.  By construction, the sets $U_i$ decrease in size, and therefore the process terminates eventually. Suppose that $N'$ steps are required.  Let $g = g_{N'-1} \circ \cdots \circ g_0 \circ \theta$.  It is evident that $g$ satisfies the required properties.

It is not difficult to see $f = g \circ h$ satisfies the conditions of the theorem.
\end{proof}

\begin{proof}[Proof of Theorem~\ref{thm:OddDegreeBelyi}]
Let $g: X \to \mathbb{P}^1$ be a non-constant meromorphic function that is defined over $\BAR{Q}$ and has odd degree, and let $S = B(g)$.  Because $g$ is defined over $\BAR{Q}$ we may assume that $S \subseteq \mathbb{P}^1_\BAR{Q}$.  By Proposition~\ref{prop:collapseUsingAnOddDegree}, there exists a function $f: \mathbb{P}^1 \to \mathbb{P}^1$ that has odd degree such that $f(S) \subseteq \{0,1,\infty\}$ and $B(f) \subseteq \{0,1,\infty\}$.  The morphism $f \circ g$ has odd degree and is unbranched outside $\{0,1,\infty\}$ by construction.
\end{proof}

\subsection{Proof of Theorem~\ref{thm:sqrtIsPowerful}}
\label{subsec:sqrtIsPowerfulProof}

The idea of the proof is to pull back Belyi maps $f$ of odd degree by $t = \frac{4f}{(f+1)^2}$ and apply Theorem~\ref{thm:PropOfSqct}(b) to constrain $\Sqrt(f)$.  In order to be able to apply Theorem~\ref{thm:PropOfSqct}(b), we need to constrain the monodromy of $f$, which we do by post-composing $f$ with a fixed Belyi map of degree 5.

Let $t_0$ denote a Belyi function with monodromy of cycle type $221$ over 0 and $\infty$ and monodromy of cycle type 5 over 1, normalized so that $t_0^{-1}(0) = \{0,1,\infty\}$ and $t_0$ is unramified at $\infty$.
It is not difficult to see that such a $t_0$ exists (for example, by Edmonds, Kulkarni, and Stong's result: Theorem~\ref{thm:Existence}).  The particular choice of which point among $\{0,1,\infty\}$ is not a ramification of $t_0$ is irrelevant.

The construction of the functions $f$ in Theorem~\ref{thm:sqrtIsPowerful} will use the following proposition.

\begin{prop}
\label{prop:cruxOfSqrtBeingPowerful}
Let $g_0$ be Belyi function of odd degree $k$.  Let $g = t_0 \circ g_0$.
\begin{enumerate}[label=(\alph*)]
\item The function $g$ is a Belyi function with monodromy of cycle type $(5^k)$ and $(2^{2k},1^k)$ over $1$ and $\infty$, respectively.
\item In the notation of Section~\ref{subsec:GQInvariance}, the morphism $f = \Sigma(g)$ is Belyi, has odd degree, and satisfies $\Sqrt(f) = \{(\sigma_0,\sigma_1,\sigma_\infty)\}$, where $(\sigma_0,\sigma_1,\sigma_\infty)$ is the monodromy triple of $g$ (defined up to simultaneous conjugation in $S_n$).
\end{enumerate}
\end{prop}
\begin{proof}
Because $t_0(\{0,1,\infty\}) = \{0\}$ and $g_0$ is Belyi, the morphism $g$ is Belyi as well.
The computation of the monodromy cycle types of $g$ over 1 and $\infty$ follow from the fact that $g_0$ is unbranched over $t_0^{-1}(\{1,\infty\})$, and part (a) follows.

It remains to prove part (b).  The fact that $f$ is unbranched outside $\{0,1,\infty\}$ follows from Proposition~\ref{prop:MonodromyProduct}.  By Propositions~\ref{prop:RepProduct}(a) and~\ref{prop:RepFiberedProduct}, the monodromy representation of $f$ acts transitively on the fiber above the base point, from which it follows that the domain of $f$ is irreducible.  Therefore, $f$ is a Belyi function.  It is evident that $f$ and $g$ have the same degree, and hence $\deg f = 5k$ is odd.

By definition, we have $g \in \Sqrt'(f)$, and it follows that $(\sigma_0,\sigma_1,\sigma_\infty) \in \Sqrt(f)$.  Note that $\sigma_\infty$ has $k$ parts of size $5$ and no parts of size 10.  Theorem~\ref{thm:PropOfSqct}(b) implies that $|\Sqct(f)| = 1$.  Because $|\Sqrt(f)| = |\Sqct(f)|$, the proposition follows.
\end{proof}

We are now ready to conclude the proof of Theorem~\ref{thm:sqrtIsPowerful}.

\begin{proof}[Proof of Theorem~\ref{thm:sqrtIsPowerful}]
Let $\sigma \not= 1 \in \GQ$.  There exists $\phi \in \BAR{Q}$ such that $\phi^\sigma \not= \phi$.  Let $E$ be a curve over $\BAR{Q}$ of genus 1 with $j(E) = \phi$.  We then know that $E^\sigma \not\cong E$ because $j\left(E^\sigma\right) = j(E)^\sigma \not= j(E)$.  Because $E$ admits a non-constant meromorphic function of degree 3 that is defined over $\BAR{Q}$, the curve $E$ admits a Belyi function $g_0: E \to \mathbb{P}^1$ of odd degree by Theorem~\ref{thm:OddDegreeBelyi}.
Let $g = t_0 \circ g_0$, let $f = \Sigma(g)$ (in the notation of Section~\ref{subsec:GQInvariance}), and let $(\sigma_0,\sigma_1,\sigma_\infty)$ be the monodromy triple of $g$.
Because $E$ is the domain of $g$ and $E^\sigma \not\cong E$ , we have $g^\sigma \not\cong g$.  By Proposition~\ref{prop:cruxOfSqrtBeingPowerful}, the degree of $f$ is odd.  Proposition~\ref{prop:cruxOfSqrtBeingPowerful} also implies that $\Sqrt(f) = \{(\sigma_0,\sigma_1,\sigma_\infty)\}$, from which it follows that $\Sqrt(f)^\sigma \not= \Sqrt(f)$.
\end{proof}

\section{Concluding remarks and open problems}
\label{sec:Conclusion}

\subsection{Generalizing the square-root class}
Let $t: \mathbb{P}^1_f \to \mathbb{P}^1_t$ be a morphism
of curves satisfying $t(\{0,1,\infty\}) \subseteq \{0,1,\infty\}$.
Given a Belyi function $f: X \to \mathbb{P}^1$, we can form the
\emph{generalized square-root class} of $f$, defined by
\[\Sqrt_t(f) = \{\text{Belyi functions }g: X' \to \mathbb{P}^1 \mid g \times_{\mathbb{P}^1_t} t \cong f\}.\]
It is clear that if $t$ is defined over a number field $K$, then
the function $\Sqrt_t(f)$ is $\Gal\left(\BAR{Q}/K\right)$-equivariant.
We recover the ordinary square-root class for the choice of $t = \frac{4f}{(f+1)^2}$.

However, if $t$ is of degree greater than 1, then $\Sqrt_t(f)$ will be empty
for most Belyi functions $f$, and therefore we do not recover a very general invariant.
In our case, where $t = \frac{4f}{(f+1)^2}$, the monodromy cycle types of $f$
above $0$ and $\infty$ must be the same in order for $\Sqrt(f)$ to be nonempty.
We give an example that suggests that one may be able to reformulate the invariant
in a manner that is applicable more generally.

\subsection{Example: Belyi functions with monodromy of cycle type \texorpdfstring{$(n,(2g+1)11\cdots1,n)$}{(n,(2g+1)11...1,n)}}
\label{sec:Example}

We apply the Orbit-Splitting Theorem to the case of Belyi functions with monodromy of
cycle type $(n,(2g+1)11\cdots1,n)$.
An explicit count of $M(n,(2g+1)11\cdots)$ and an application of the $n$-cycle Orbit-Splitting Theorem~\ref{thm:OrbitSplitting}
yield the following result.

\begin{prop}
\label{prop:nn2g+1}
Let $g$ be a positive integer and let $n \ge 4g + 1$ be an odd positive integer.  Then,
there are at least $\left\lfloor\left(\frac{g}{2}+1\right)^2\right\rfloor$ $\GQ$-orbits
classes of Belyi maps with monodromy of cycle type $(n,(2g+1)11\cdots1,n).$
\end{prop}

In the case of $g = 1$ and $n = 5,7,9$, we constructed the Belyi functions and explicitly
verified the following conjecture, which suggests that the square-root cycle type class
can be adapted to an invariant that describes the combinatorial action of $\GQ$ on the groups
of divisors or principal divisors.

\begin{conj}
Let $n$ be an odd positive integer, $X$ an algebraic curve, and $f: X \rightarrow \mathbb{P}^1$
a Belyi function with monodromy of cycle type $(n,311\cdots1,n)$.  Let $P$ and $O$
be the locations of the ramifications of order $n-1$ on $X$, and let $T$ be
the location of the ramification of order 2.  Then,
$\Sqct(f) = \{(22\cdots 2111,322\cdots 2,n)\}$ if and only if
\[(T) \sim \frac{n+1}{2} (P) - \frac{n-1}{2} (O)\]
as divisors on $X$.
\end{conj}

\appendix

\section{Proofs of Lemmata~\ref{lemma:BoundN_0(t)}, \ref{lemma:BoundSumf}, and~\ref{lemma:BoundProdf}}
\label{app:CompLemmata}

\begin{proof}[Proof of Lemma~\ref{lemma:BoundN_0(t)}]
We have \begin{align*}
n &\le 2t + 1 + \sum_{k=1}^t 2(2k-1)\left(\frac{4t+2}{2k-1}-1\right)
= 2t + 1 + 2\sum_{k=1}^t (4t+3 - 2k)\\
&= 2t+1 + t(6t+6)
= 6t^2+12t+1 < 6(t+1)^2\end{align*}
and
\begin{align*}
n &> 2t + 1 + \sum_{k=1}^t 2(2k-1)\left(\frac{4t+2}{2k-1}-2\right)
= 6t^2+12t+1 - \sum_{k=1}^t 2(2k-1)\\
&= 4t^2 + 12t + 1.\end{align*}
\end{proof}

\begin{proof}[Proof of Lemma~\ref{lemma:BoundSumf}]
We have
\[\sum_{k=1}^t (f(k)-1) \le \sum_{k=1}^t \left(\frac{4t+1}{2k-1}-1\right)
=-t + (4t+1) \sum_{k=1}^t \frac{1}{2k-1}.\]
Applying the bound
\[\log(m+1) \le \sum_{k=1}^m \frac{1}{k} \le \log m + 1,\]
which holds for all positive integers $m$, we have
\[\sum_{k=1}^t (f(k)-1) \le -t + (4t+1)\left(\log(2t-1) + 1 - \frac{\log(t)}{2}\right).\]
Therefore, we have
\[2\sum_{k=1}^t (f(k)-1) \le 6t + 2 + (4t+1)\log(4t).\]
It follows that $2\sum_{k=1}^t (f(k)-1) \le 2t^2+6t+\frac{1}{2} \le \frac{n_0(t)}{2}$ for $t \ge 8$,
where the second inequality is by Lemma~\ref{lemma:BoundN_0(t)}.
We can easily verify the lemma for $t \le 7$, and the lemma follows.
\end{proof}

\begin{proof}[Proof of Lemma~\ref{lemma:BoundProdf}]
Fix $t$, and let $M$ denote the left-hand side.  We have
\[M > \prod_{k=1}^t \left(\frac{4t+2}{2k-1} - 1\right)
= \frac{\prod_{k=1}^t (4t + 3 - 2k)}{\prod_{k=1}^t (2k-1)}.\]
Recall that
\[(2m-1)!! = \prod_{k=1}^m (2k-1) = \frac{(2m)!}{2^m(m!)}.\]
Returning to $M$, we have
\begin{align*}
M &> \frac{(4t+1)!!}{(2t+1)!!(2t-1)!!}
= \frac{(4t+2)!2^{t+1}2^t}{(2t+2)!(2t)!2^{2t+1}}
= \frac{(4t+2)!2^{t+1}(t+1)!2^t t!}{(2t+1)!(2t+2)!(2t)!2^{2t+1}}\\
&= \frac{(4t+2)!(t+1)!t!}{(2t+1)!(2t+2)!(2t)!} = \frac{\binom{4t+2}{2t+1}}{2\binom{2t}{t}}.\end{align*}
We now apply Stirling's formula with error bounds, which is the well-known inequality
\[e^{\frac{1}{12m+1}} < \frac{m!}{\sqrt{2\pi m} \left(\frac{m}{e}\right)^m} < e^{\frac{1}{12m}}.\]
It follows that
\[e^{\frac{1}{24m+1}-\frac{1}{6m}} < \frac{\binom{2m}{m}\sqrt{\pi m}}{2^m} < e^{\frac{1}{24m}-\frac{2}{12m+1}}.\]
In particular, we have
\[\frac{-1}{6m}<\log \frac{\binom{2m}{m}\sqrt{\pi m}}{2^{2m}} <0.\]
Applying this bound to $M$, we have
\[M > 2^{2t}\sqrt{2} e^{\frac{-1}{12t+6}}
> 2^{2t}.\]
\end{proof}

\bibliographystyle{abbrv}
\bibliography{commsquare}

\end{document}